\documentclass{article}
\usepackage{amsmath}

\usepackage{xspace}

\usepackage[boxed, linesnumbered, ruled]{algorithm2e}
\usepackage{amsfonts}
\usepackage{graphicx}
   \usepackage{amsmath, amssymb, amsthm}
\usepackage{mathrsfs}

\usepackage{mathtools}
\usepackage{tikz}
 \newtheorem{thm}{Theorem}[section]
 \newtheorem{cor}[thm]{Corollary}
 \newtheorem{lem}[thm]{Lemma}
 
 \newtheorem{prop}[thm]{Proposition}
 \newtheorem{obs}[thm]{Observation}
 \theoremstyle{definition}
 \newtheorem{defn}[thm]{Definition}
 \newtheorem{exmp}[thm]{Example}

\numberwithin{equation}{section}

\definecolor{cyellow}{RGB}{255,246,177}

\newtheorem{claim}{Claim}
\newtheorem{claimproof}{Proof of Claim}

\begin{document}
\date{}
\title {On metric dimension of cube of trees}
\author{Sanchita Paul\thanks{Indian Statistical Institute, Kolkata, India. sanchitajumath@gmail.com}\thanks{(Corresponding author)},
Bapan Das\thanks {Department of Mathematics, Balurghat College, Balurghat 733101, India. mathbapan@gmail.com}
Avishek Adhikari \thanks{Department of Mathematics, Presidency University, Kolkata 700073, India. avishek.maths@presiuniv.ac.in},
Laxman Saha \thanks{Department of Mathematics, Balurghat College, Balurghat 733101, India. laxman.iitkgp@gmail.com}}

\maketitle

\begin{abstract}
\noindent
{\footnotesize Let $G=(V,E)$ be a connected graph and $d_{G}(u,v)$ be the shortest distance between the vertices $u$ and $v$ in $G$. A set $S=\{s_{1},s_{2},\hdots,s_{n}\}\subset V(G)$ is said to be a {\em resolving set} if for all distinct vertices $u,v$ of $G$, there exist an element $s\in S$ such that $d(s,u)\neq d(s,v)$. The minimum cardinality of a resolving set for a graph $G$ is called the {\em metric dimension} of $G$ and it is denoted by $\beta{(G)}$. A resolving set having $\beta{(G)}$ number of vertices is named as {\em metric basis} of $G$. The metric dimension problem is to find a metric basis in a graph $G$, and it has several real-life applications in network theory, telecommunication,
image processing, pattern recognition, and many other fields. In this article, we
consider {\em cube of trees} $T^{3}=(V, E)$, where any two vertices $u,v$ are adjacent if and only if the distance between them is less than equal to three in $T$. 
We establish the necessary and sufficient conditions of a vertex subset of $V$ to become a resolving set for $T^{3}$. This helps determine the tight bounds (upper and lower) for the metric dimension of $T^{3}$. Then, for certain well-known cubes of trees, such as caterpillars, lobsters, spiders, and $d$-regular trees, we establish the boundaries of the metric dimension. Further, we characterize some restricted families of cube of trees satisfying 
$\beta{(T^{3})}=\beta{(T)}$. We provide a construction showing the existence of a cube of tree attaining every positive integer value as their metric dimension.
}
\end{abstract}



\section{Introduction}
\noindent For a simple undirected connected graph $G=(V,E)$, the length of the shortest distance between the vertices $u$ and $v$ in $G$ is denoted 
by $d_{G}(u,v)$. Instead of $d_{G}(u,v)$, we use $d(u,v)$ if $G$ is already predefined. The {\em code} of a vertex $w$ with respect to a vertex set $S=\{s_{1},\hdots,s_{n}\}\subseteq V$ (denoted by $c(w|S)$) is a $n-$tuple $(d(w,s_{1}),\hdots,d(w,s_{n}))$. A vertex $s$ {\em resolves} two distinct vertices $u,v$ of $V$ when $c(u|s)\neq c(v|s)$, i.e., $d(u,s)\neq d(v,s)$ considering $S=\{s\}$. In the same sense, $S$ is said to be a {\em resolving set} for $G$ if for every two distinct vertices $u,v$ of $V$, we have $c(u|S)\neq c(v|S)$, i.e., for any such $u\neq v$ there exists a vertex $s\in S$ which resolves $u,v$. If no $s\neq u,v$ is found to satisfy the above criteria, then we include one among $u$ or $v$ in $S$ (cf. for a $n$-vertex complete graph $K_{n}$, $S$ contains $n-1$ vertices in it). The smallest possible resolving set is said to be {\em metric basis} and its cardinality is called {\em metric dimension} of the graph $G$ (in short dim${(G)}$). For convenience, $\beta{(G)}$ is used to denote the metric dimension of a graph $G$. The metric basis does not need to be unique for a given graph $G$.

\vspace{0.5em}
 The problem of determining metric dimension is NP-complete for many restricted classes of graphs such as planer graphs, split graphs, bipartite and co-bipartite graphs, line graphs of bipartite graphs, etc \cite{GJ}. Finding the metric basis of connected graphs was introduced independently by Slater \cite{S} and Harary and Melter \cite{HM} in 1975 and 1976, respectively for uniquely identifying every vertex in a graph. They found a polynomial-time characterization for the metric dimension of trees. After that, Khuller et al. \cite{K} gave a similar characterization for the metric dimension of trees and developed a linear time algorithm for obtaining the metric basis (they call it landmarks). Sometimes the elements of a metric basis are treated as sensors \cite{CZ} in a real-world network to preserve system security by transferring information or messages within a fixed group only. Finding such a minimal group (landmark) is also crucial in robot navigation problem \cite{K} where the robot can uniquely determine its position by the presence of distinctly labeled landmarks. For more extensive applications of metric basis in various fields, such as optimization, network discovery, telecommunication, geographical routing protocols, image processing, pattern recognition, chemistry, and others, one may see \cite{BEE, CEJO, JRA, KRR, LG, MT, St}. Recently, it has been proven that the metric dimension is FPT parameterized by treewidth on chordal graphs \cite{BDP}.

\vspace{0.5em}
The power graph has been extensively explored in the past due to its intriguing features and wide range of applications in routing in networks, quantum random walk in physics, etc. Alholi et al. \cite{AAE} have determined the upper bound for the power of paths. In 2021, Nawaz et al. \cite{NAKK} proved that the metric dimension of path power three and four is unbounded; they also proved some results on the edges of the power of path and power of total graph. Saha et al. \cite{Saha} presented a lower bound for the metric dimension of $P_{n}^{r}$ and then built up a resolving set with cardinality that is the same as that of the lower bound. Also, they have investigated the bounds of metric dimension for the square of trees. Due to the widespread applications of power graphs and motivated by the above results, in this article, we study {\em cube of trees} $T^{3}=(V, E)$ where any two vertices $u,v\in V$ are adjacent if and only if $d_{T}(u,v)\leq 3$.

\vspace{0.5em}
The rest of the paper is organized as follows: Firstly, Section $2$ represents a detailed explanation of all the terms and expressions that will be used later on to establish the corresponding results for the metric dimension of $T^{3}$. In Section $3$, we have proved some essential lemmas on the properties of resolvability in $T^{3}$ that facilitate determining the resolving set of the cube of a tree. In Section $4$, first, we provide the necessary and sufficient conditions of a vertex subset of $V$ to become a resolving set for $T^{3}$. Next in Section $5$ and in Section $6$, we build tight bounds (lower and upper) for $\beta{(T^{3})}$ depending upon the number of {\em short legs}, {\em long legs}, {\em major stems} and their positional appearance in the tree $T$. It is a worthy task to construct a resolving set for proving the upper bound of $T^{3}$.
Furthermore, we provide a construction showing the existence of a family of cube of trees attaining every positive integer as their metric dimension. In Section $7$, we analyze the metric dimension or the bounds of it for some well-known cube of trees, including caterpillars, lobsters, spiders, $d$-regular trees. Lastly, in Section $8$, we restrict our findings to those cube of trees that have pendants as their legs and all of their stems lie on a central path and characterize such graph classes that satisfy $\beta{(T^{3})=\beta{(T)}}$.  In the conclusion section, we keep the challenge open to determine the bounds of metric dimension for any power of trees $T^{r}$ (say) where $r\geq 4$.

\section{Preliminaries}
  For a tree $T=(V, E)$, a vertex $v\in V$ of degree at least three is called {\em core vertex} or {\em core}, a vertex of degree two and one is said to be {\em path vertex} and {\em leaf} respectively  \cite{AE}. If we remove a vertex $v$ from $T$ then $T\setminus\{v\}$ induces a deg$(v)$ number of subtrees or components.
A {\em branch} at a vertex $v$ is the subgraph induced by $v$ and one of the components of $T\setminus \{v\}$. A branch $B$ of $T$ at $v$ which is a path is called {\em branch path} (also known as {\em leg}) \cite{S}. The vertex $v$ in a branch path satisfying deg$(v)\geq 3$ is called {\em stem} of the branch path \cite{HMSW}. It is easy to observe that not every core vertex is a stem.
  
\begin{defn}\label{majorstem}
A vertex of a tree $T=(V, E)$ is said to be a {\em major stem} if it is a stem containing at least two legs. Other stems are called {\em minor stems}. A leg of length greater than or equal to three is said to be {\em long leg}, other legs that have a length less than three are said to be {\em short legs}. We call a short leg of length two as {\em mid leg} and a short leg of length one as {\em pendant}. 
\end{defn}

\begin{obs}
Let $T=(V, E)$ contain at least one stem. Then the following are true:

$i)$ Two legs adjacent to the same stem vertex $v\in V$ are disjoint except for the common stem $v$.

$ii)$ Any two legs adjacent to two distinct stems must be disjoint. 
\end{obs}

\begin{thm} \cite{S}\label{component}
Let $T=(V,E)$ be a tree of order $|V|\geq 3$. Then $S\subseteq V$ forms a resolving set if and only if for each vertex $x$ there are vertices from $S$ on at least deg$(x)-1$ of the deg$(x)$ components of $T\setminus\{x\}$.
\end{thm}

\noindent The problem of computing the metric dimension of trees was solved in linear time by Khuller et al. \cite{K} in 1996.

\begin{thm}\cite{K}\label{0} 
Let $T=(V, E)$ be a tree that is not a path. If $l_{v}$ is the number of legs attached to the vertex $v$. Then
\begin{equation}
\beta{(T)}= \sum_{\substack{v\in V:l_{v}>1}} (l_{v}-1)
\end{equation}                  
\end{thm}

\noindent As the minor stems of a tree cannot have more than one leg as its branch, it is important to note the following from Theorem \ref{0}. 

\begin{cor}\label{6}
Let $T$ be a tree that is not a path. Then $\beta{(T)}=\sum\limits_{\substack{v\in V^{\prime}}}(l_{v}-1)$ where $V^{\prime}$ denotes the set of all major stems of $T$ and $l_{v}$ is the number of legs attached to the major stem $v$.
\end{cor}

\vspace{0.3em}

\noindent \textbf{Notation.}  Let $P=(u,u_{1},\hdots, v_{1},v)$ be the path on a tree $T$ between the vertices $u$ and $v$. Here $u_{1},v_{1}$ are either the intermediate vertices of the above path $P$ considering $d_{T}(u,v)\geq 2$ ($u_{1}$ can be equal to $v_{1}$ also when $d_{T}(u,v)=2$) or end vertices when $d_{T}(u,v)=1$ (i.e., $u_{1}=v_{1}=v$ or $u=u_{1}=v_{1}$ or $u=u_{1},v=v_{1}$). We denote $T_{u} (T_{v})$ to be the component of $T$ containing the vertex $u (v)$, obtained after deletion of the edge $uu_{1}$($v_{1}v$). A vertex $x$ is said to be within the same component of $u$ and $v$ (say $T_{u,v}$) only when $x$ occurs within the intermediate path of $u,v$ or it lies in some branch of $T$ attached to some intermediate vertex of $u,v$.

\begin{defn}\label{cube of tree}
Let $T=(V,E)$ be a tree. A graph $T^{3}=(V,\hat{E})$ is said to be {\em cube of tree} of $T$ if the vertex set $V$ remains same as in $T$ and the edge set 
$\hat{E}=E\cup \{uv \hspace{0.1em}|\hspace{0.1em}  2\leq d_{T}(u,v)\leq 3\}$. 
\end{defn}

\noindent The distance between any two vertices $u,v$ in $T^{3}$ is measured by $d_{T^{3}}(u,v)=\left\lceil \dfrac{d_{T}(u,v)}{3} \right\rceil$. We will use the notations 
$V(T^{3})$ and $E(T^{3})$ to denote the vertex set and edge set of $T^{3}$.

\section{Properties regarding resolvability in $T^{3}$}
In this section, we give some basic properties and results of the resolving set of $T^{3}$. We have established certain essential lemmas that are beneficial for determining the resolving set of $T^{3}$.

\begin{lem}\label{1}
Let $T=(V,E)$ be a tree. Then every resolving set of $T^{3}$ is also a resolving set of $T$.  
\end{lem}

\begin{proof}
Let $S$ be a resolving set of $T^{3}$ and $u,v\in V$ be any two vertices. Since $V(T^{3})=V(T)$ and $S$ is a resolving set of $T^{3}$, there exists a vertex $s\in S$ such that $d_{T^{3}}(s,u)\neq d_{T^{3}}(s,v)$, which imply $\lceil \dfrac{d_{T}(s,u)}{3} \rceil \neq \lceil\dfrac{d_{T}(s,v)}{3} \rceil$. 
Hence we get $d_{T}(s,u)\neq d_{T}(s,v)$, i.e., $s$ resolves the vertices $u$ and $v$ in $T$. Therefore, $S$ forms a resolving set for $T$.
\end{proof}

\noindent We can immediately draw some conclusion from the above lemma.
\begin{cor}\label{1.5}
For any tree $T$, $\beta{(T^{3})}\geq \beta{(T)}$. 
\end{cor}

\begin{proof}
Let $S$ be a metric basis for $T^{3}$. Then $|S|=\beta{(T^{3})}$. Using Lemma \ref{1}, we get $S$ to be a resolving set of $T$ also. Therefore, $\beta{(T)}\leq |S|=\beta{(T^{3})}$.
\end{proof}

\begin{lem}\label{2}
Let $T=(V,E)$ be a tree and $S$ be a resolving set of $T^{3}$. Then for every vertex $x\in V$, $S$ contains a vertex from each component of $T\setminus \{x\}$ with one exception.
\end{lem}

\begin{proof}
On the contrary, let $T\setminus \{x\}$ has at least two components (say $C_{i},C_{j}$) satisfying $S\cap V(C_{i})=\emptyset$ and $S\cap V(C_{j})=\emptyset$. Let $u\in S\cap V(C_{i})$ and $v\in S\cap V(C_{j})$ satisfy $d_{T}(x,u)=d_{T}(x,v)$. Now any vertex $w\in V\setminus (V(C_{i})\cup V(C_{j}))$ must have to reach $u$ or $v$ via $x$. Therefore, $d_{T}(w,u)=d_{T}(w,x)+d_{T}(x,u)=d_{T}(w,x)+d_{T}(x,v)=d_{T}(w,v)$ and hence $d_{T^{3}}(w,u)=\lceil \dfrac{d_{T}(w,u)}{3}\rceil=\lceil \dfrac{d_{T}(w,v)}{3}\rceil=d_{T^{3}}(w,v)$. Therefore, a contradiction arises. Hence, the result follows.
\end{proof}

\noindent The following corollary is an essential tool for determining any resolving set of $T^{3}$. 

\begin{cor}\label{30}
Let $v$ be a major stem of a tree $T$ having $m$ legs $L_{1},L_{2},\hdots, L_{m}$. Then, for every resolving set $S$ of $T^{3}$, the following holds true.

\begin{enumerate}
\item $S\cap L_{i}\neq \emptyset$ for all $i\in \{1,2,\hdots,m\}$ with one exception.

\item $S$ contains at least $m-1$ vertices from the legs adjacent to $v$.
\end{enumerate}
\end{cor}

\begin{lem}\label{5}
Let $T=(V, E)$ be a tree, and $v\in V$ be a core of degree $m$. If $v$ is not a major stem, then there exist at least $m-1$ components of $T$ containing major stems.  
\end{lem}
  
\begin{proof}
Since deg $(v)=m$, removing $v$ from $T$ will create $m$ components. Now, as $v$ is not a major stem, there can exist at most one branch attached to it, which is a path. Hence, there are $m-1$ branches containing at least one vertex in each of the branches, which have at least two branches out from them. Each of these $m-1$ branches is not the path. We consider one such branch $B$ of $v$ and a vertex $u$ on $B$ having deg $(u)\geq 3$ for which $d_{T}(v,u)$ is maximum. Therefore, one can verify that $u$ must possess at least two branch paths, and hence $u$ becomes a major stem of $B$, as well as of $T$ from Definition \ref{majorstem}. Similar logic holds true for all other branches of $v$ that are not paths. Hence, the result follows.
\end{proof}

\begin{cor}\label{disjoint}
Let $v$ be a core vertex of a tree $T$ having $m$ components $C_{1}, \hdots, C_{m}$. If any component $C_{i}$ contains $l_{i}$ major stems where $1\leq i\leq m$, then for every resolving set $S$ of $T^{3}$, $|S \cap C_{i}|\geq \sum \limits_{\substack{j=1}}^{l_{i}} (n_{j}-1)$, where $n_{j}$ is the number of legs attached to a major stem in $C_{i}$.
\end{cor}

\begin{proof}
By Lemma \ref{5}, it follows that at least $m-1$ components among $C_{1},\hdots, C_{m}$ contain major stems. If $C_{i}$ is not a branch path, then applying Lemma \ref{2} for each major stem of $C_{i}$ it follows that
$S$ contains at least $\sum\limits_{\substack{{j=1}}}^{l_{i}} (n_{j}-1)$ vertices from the legs adjacent to the major stems of $C_{i}$.
\end{proof}

\begin{lem}\label{3}
Let $T=(V,E)$ be a tree and $uv$ be an edge in $T^{3}$. Then a vertex $x\neq u,v$ resolves $u,v$ in $T^{3}$ if and only if the following happens.

\begin{itemize}
\item If $x$ belongs to at least one among $T_{u}$ or $T_{v}$ then either 

\begin{center}
$d_{T}(u,v)=3$ or\\
\vspace{0.2em}
$d_{T}(u,v)=2$ and $\text{min}\{d_{T}(x,u),d_{T}(x,v)\}\equiv 0  \hspace{0.3 em}or \hspace{0.3 em} 2\hspace{0.3 em}(\text{mod} \hspace{0.3 em} 3)$ or \\
 
\vspace{0.2em}
$d_{T}(u,v)=1$ and $\text{min}\{d_{T}(x,u),d_{T}(x,v)\}\equiv 0\hspace{0.3 em} (\text{mod} \hspace{0.3 em} 3)$
\end{center}

\item If $x$ belongs to $T_{u,v}$ then $d_{T}(u,v)=3$ and 
 $\text{min}\{d_{T}(x,u),d_{T}(x,v)\}\equiv \hspace{0.3 em} 0 \hspace{0.3 em}(\text{mod}  \hspace{0.3 em}3)$. 
 \end{itemize}
\end{lem}


\begin{proof} 
Since $uv\in E(T^{3})$, $1\leq d_{T}(u,v)\leq 3$ clearly.

\vspace{0.2em}
\noindent \textbf{Case I:} Without loss of generality, first we consider the case when $x\in T_{u}$. Then we can write $d_{T}(x,u)=3k+m$ and $d_{T}(x,v)=d_{T}(x,u)+d_{T}(u,v)=(3k+m)+d_{T}(u,v)$ for some integers $k,m$ where $k\geq0$, $0\leq m<3$. Hence $\text{min}\{d_{T}(x,u),d_{T}(x,v)\}=d_{T}(x,u)$.  
\vspace{0.2em}

 If $m=0$, $d_{T^{3}}(x,u)=\lceil{\dfrac{3k}{3}}\rceil=k\neq k+1=\lceil{\dfrac{3k+d_{T}(u,v)}{3}}\rceil= d_{T^{3}}(x,v)$ as $1\leq d_{T}(u,v)\leq 3$. Therefore, when $\text{min}\{d_{T}(x,u),d_{T}(x,v)\}=3k\equiv 0\hspace{0.3 em} (\text{mod} \hspace{0.3 em} 3)$, then $x$ resolves $u,v$.

\vspace{0.3em}
For $m=1$ or $2$,  $d_{T^{3}}(x,u)=\lceil{\dfrac{3k+m}{3}}\rceil=k+1$ and $d_{T^{3}}(x,v)=\lceil \dfrac{(3k+m)+d_{T}(u,v)}{3} \rceil=k+\lceil \dfrac{m+d_{T}(u,v)}{3}\rceil$. Now $x$ resolves $u,v$ if and only if $d_{T^{3}}(x,v)=k+2$ (since $d_{T}(u,v)\leq 3$). This can only happen when $m+d_{T}(u,v)>3$, i.e., when $m=1$ and $d_{T}(u,v)=3$ or when $m=2$ and $2\leq d_{T}(u,v)\leq 3$. Therefore, if $d_{T}(u,v)=3$ and $\text{min}\{d_{T}(x,u),d_{T}(x,v)\}\equiv 1 \hspace{0.3 em} \text{or} \hspace{0.3 em} 2\hspace{0.3 em} (\text{mod} \hspace{0.3 em} 3)$ or if  $d_{T}(u,v)=2$ and $\text{min}\{d_{T}(x,u),d_{T}(x,v)\}\equiv 2 \hspace{0.3 em} (\text{mod} \hspace{0.3 em} 3)$ then $x$ resolves $u,v$.

\vspace{0.3em}
\noindent \textbf{Case II:} Next, we consider the case when $x$ belongs to the same component of $u$ and $v$, i.e., in $T_{u,v}$. Since $x\neq u,v$, $d_{T}(u,v)>1$. Note that in this case, the only possibility of $x$ resolving $u,v$ is when $d_{T}(u,v)=3$ and $x$ occurs in some branch attached to $u_{1}$ or $v_{1}$ where $P=(u,u_{1},v_{1},v)$ is the path connecting $u,v$ in $T$. Without loss of generality, we assume $\text{min}\{d_{T}(x,u),d_{T}(x,v)\}=d_{T}(x,u)$. Then $x$ must be attached to the branch of $u_{1}$. Let $d_{T}(x,u_{1})=3k+m$ for some nonnegative integers $k,m$ satisfying $0\leq m<3$. Then $d_{T}(x,u)=d_{T}(x,u_{1})+d_{T}(u_{1},u)=(3k+m)+1$ and $d_{T}(x,v)=d_{T}(x,u_{1})+d_{T}(u_{1},v)=(3k+m)+2$. Therefore, $d_{T^{3}}(x,u)=k+\lceil \dfrac{m+1}{3}\rceil$ and $d_{T^{3}}(x,v)=k+\lceil \dfrac{m+2}{3}\rceil$. One can easily verify now that $x$ resolves $u,v \Longleftrightarrow d_{T^{3}}(x,u)\neq d_{T^{3}}(x,v)\Longleftrightarrow m=2$. Therefore, $\text{min}\{d_{T}(x,u),d_{T}(x,v)\}=d_{T}(x,u)=3k+3\equiv 0 \hspace{0.3em} (\text{mod} \hspace{0.3em} 3)$. 
\end{proof}  

If $uv$ is an edge in $T^{3}$, then depending upon the different values of $d_{T}(u,v)$ we can impose restrictions on the vertices that can resolve $u,v$. 

\begin{cor}\label{distance}
Let $T=(V,E)$ be a tree and $uv$ be an edge in $T^{3}$. Then a vertex $x\neq u,v$ resolves $u,v$ in $T^{3}$ if and only if the following are true:
\begin{enumerate}

\item  If $d_{T}(u,v)=1$, then at least one among any three consecutive vertices chosen from $T_{u}\setminus\{u\}$ or $T_{v}\setminus\{v\}$ must coincide with $x$. \label{one distance}

\item If $d_{T}(u,v)=2$, then $x$ must be one among any two consecutive vertices chosen from $T_{u}\setminus \{u\}$ or $T_{v}\setminus\{v\}$.\label{two distance}

\item If $d_{T}(u,v)=3$ then $x$ is either in $T_{u}\setminus\{u\}$ or $T_{v}\setminus\{v\}$ or it is one among any three consecutive vertices from any branch attached to $u_{1}$ or $v_{1}$ where $u_{1},v_{1}$ are the intermediate vertices of the path $(u,u_{1},v_{1},v)$ in $T$. \label{three distance}
 
\end{enumerate}
\end{cor}

\begin{proof}
It is easy to observe that the distance from a fixed vertex to any three (or two) consecutive vertices in $T^{3}$ must be different \footnote{it must be a 
$3$-permutation (or $2$-permutation) of the set $\{0,1,2\}$} computed in mod $3$. The rest of the verification is immediate from Lemma \ref{3}.
\end{proof}

\begin{lem}\label{4}
Let $T=(V,E)$ be a tree and $u,v$ be two nonadjacent vertices in $T^{3}$. Then a vertex $x\neq u,v$ resolves $u,v$ if and only if the following conditions are satisfied:

\begin{itemize}
\item $d_{T}(x,u)\neq d_{T}(x,v)$.

\item If $x$ belongs to $T_{u,v}$ then

\hspace{6em} min $\{d_{T}(x,u),d_{T}(x,v)\}\equiv 0 \hspace{0.2em} (\mbox{mod 3})$ and  $|d_{T}(x,v)-d_{T}(x,u)|\geq 1$ or

\hspace{6em} min $\{d_{T}(x,u),d_{T}(x,v)\}\equiv 1 \hspace{0.2em} (\mbox{mod 3})$ and $|d_{T}(x,v)-d_{T}(x,u)|\geq 3$ or

\hspace{6em} min $\{d_{T}(x,u),d_{T}(x,v)\}\equiv 2 \hspace{0.2em} (\mbox{mod 3})$ and $|d_{T}(x,v)-d_{T}(x,u)|\geq 2$.

\vspace{0.39em}

\item Any $x$ belonging to $T_{u}$ or $T_{v}$ can resolve $u,v$.

\end{itemize}
\end{lem}

\begin{proof}
In $T^{3}$, a vertex $x\neq u,v$ resolves $u,v$ if and only if $d_{T^{3}}(x,u)\neq d_{T^{3}}(x,v)$. This imply $\lceil \dfrac{d_{T}(x,u)}{3}\rceil\neq \lceil \dfrac{d_{T}(x,v)}{3} \rceil$ and hence $d_{T}(x,u)\neq d_{T}(x,v)$. Now as $u,v$ are nonadjacent in $T^{3}$, we have $d_{T}(u,v)>3$. Consider the two cases below. 

\vspace{0.6em}

\noindent \textbf{Case I:} First we consider the case when $x$ is in $T_{u,v}$. Let $s$ be the intermediate vertex on the path $P=(u,u_{1},\hdots,s,\hdots,v_{1},v)$ connecting the unique path joining $x$ to $s$ in $T$.
Now $d_{T}(x,u)\neq d_{T}(x,v) \Longleftrightarrow d_{T}(s,u)\neq d_{T}(s,v)$. Without loss of generality we assume  min $\{d_{T}(s,u),d_{T}(s,v)\}=d_{T}(s,u)$. Then $d_{T}(x,u)=d_{T}(x,s)+d_{T}(s,u)$ and $d_{T}(x,v)=d_{T}(x,s)+d_{T}(s,v)$ and therefore 
min $\{d_{T}(x,u),d_{T}(x,v)\}=d_{T}(x,u)$. It is easy to note that $d_{T}(x,u)\geq 2$ always.

\vspace{0.6em}
 $a)$ If $d_{T}(x,u)\equiv 0 \hspace{0.2em} (\mbox{mod} \hspace{0.2em} 3)$, then $d_{T}(x,u)=3k$ for some positive integer $k$ and $d_{T^{3}}(x,u)=k$. Since $d_{T}(x,v)>d_{T}(x,u)$, $d_{T}(x,v)\geq 3k+1$, which implies $d_{T}(x,v)-d_{T}(x,u)\geq 1$. Hence we get $d_{T^{3}}(x,v)\geq \lceil \dfrac{3k+1}{3}\rceil =k+1>k= d_{T^{3}}(x,u)$. 

\vspace{0.5em}
 $b)$ If $d_{T}(x,u)\equiv 1 \hspace{0.2em} (\mbox{mod} \hspace{0.2em} 3)$, then $d_{T}(x,u)=3k+1$ for positive integer $k$ and $d_{T^{3}}(x,u)=\lceil\dfrac{3k+1}{3}\rceil=k+1$. Since $d_{T}(x,v)>d_{T}(x,u)$, we have $d_{T}(x,v)\geq 3k+2$. Now $d_{T^{3}}(x,v)\neq d_{T^{3}}(x,u) 
\Longleftrightarrow
\lceil \dfrac{d_{T}(x,v)}{3}\rceil \neq k+1$. This implies that $d_{T}(x,v)\neq 3k+2, 3k+3$ and hence $d_{T}(x,v)\geq 3k+4$. Therefore, $d_{T}(x,v)-d_{T}(x,u)\geq 3$.

\vspace{0.7em}
 $c)$ If $d_{T}(x,u)\equiv 2 \hspace{0.2em} (\mbox{mod} \hspace{0.2em} 3)$, then
$d_{T}(x,u)=3k+2$ for some integer $k\geq 0$ and $d_{T^{3}}(x,u)=\lceil\dfrac{3k+2}{3}\rceil=k+1$. Also, $d_{T}(x,v)>d_{T}(x,u)$ implies $d_{T}(x,v)\geq 3k+3$. Now $d_{T^{3}}(x,v)\neq d_{T^{3}}(x,u)
\Longleftrightarrow
\lceil \dfrac{d_{T}(x,v)}{3}\rceil \neq k+1$. This implies that $d_{T}(x,v)\neq 3k+3$ and hence $d_{T}(x,v)\geq 3k+4$. Therefore, $d_{T}(x,v)-d_{T}(x,u)\geq 2$.

\vspace{0.6em}

\noindent  If $x$ is an intermediate vertex of the $u-v$ path $P$, then considering $s=x$ the similar logic will follow. 

\vspace{0.6em}

\noindent \textbf{Case II:} Next, we consider the case when $x$ is either in $T_{u}$ or $T_{v}$. Without loss of generality, we assume that $x$ is in $T_{u}$. Then $d_{T}(x,v)=d_{T}(x,u)+d_{T}(u,v)>d_{T}(x,u)+3$ as $u,v$ are nonadjacent in $T^{3}$. Therefore $d_{T^{3}}(x,v)>d_{T^{3}}(x,u)+1$. Hence, any such $x$ can resolve $u,v$.
\end{proof}

\begin{cor}\label{four distance}
Let $T=(V,E)$ be a tree and $u,v$ be two nonadjcent vertices
in $T^{3}$ satisfying $4\leq d_{T}(u,v)\leq 5$. Then a vertex $x\neq u,v$ resolves $u,v$ in $T^{3}$ if and only if the following are true:

\begin{enumerate}
\item If $d_{T}(u,v)=4$. Then $x$ is either in $T_{u}\setminus\{u\}$ or $T_{v}\setminus\{v\}$ or it is one among any two consecutive vertices from any branch of $T$ attached to the intermediate vertex $u_{1}$ or $v_{1}$ of the path $(u,u_{1},w,v_{1},v)$ in $T$.

\item If $d_{T}(u,v)=5$. Then $x$ is either $T_{u}\setminus \{u\}$ or $T_{v}\setminus\{v\}$ or $x$ coincides with $u_{1}$ or $v_{1}$ or any vertex on a branch attached to them or it is one among any three consecutive vertices from any branch of $T$ attached to the intermediate vertices $w_{1}$ or $w_{2}$ of the path $(u,u_{1},w_{1}, w_{2}, v_{1},v)$ in $T$.
\end{enumerate}
\end{cor}

\begin{proof}
Let $d_{T}(u,v)=4$. Without loss of generality, we assume $x,y$ to be two consecutive vertices on a branch $B$ attached to $u_{1}$. Then min$\{d_{T}(x,u),d_{T}(x,v)\}=d_{T}(x,u)$ and min$\{d_{T}(y,u),d_{T}(y,v)\}=d_{T}(y,u)$. Now $d_{T}(x,v)-d_{T}(x,u)=(d_{T}(x,u_{1})+d_{T}(u_{1},v))-(d_{T}(x,u_{1})+d_{T}(u_{1},u))=d_{T}(u_{1},v)-d_{T}(u_{1},u)=2$ as $d_{T}(u,v)=4$. Similarly, we get $d_{T}(y,v)-d_{T}(y,u)=2$. Since the vertices $x,y$ are consecutive along $B$, at least one among $d_{T}(x,u)$ or $d_{T}(y,u)$ takes a value from the set $\{0,2\}$ computed in mod $3$. Let $d_{T}(x,u)\equiv 0$ or $2$ (mod $3$). Then, by Lemma \ref{4}, $x$ resolves $u,v$ in $T^{3}$. Similar logic follows if $d_{T}(y,u)\equiv 0$ or $2$ (mod $3$).

\vspace{0.4em}

\noindent The proof of resolvability for the case $d_{T}(u,v)=5$ is analogous and can be verified using Lemma \ref{4}.
\end{proof}

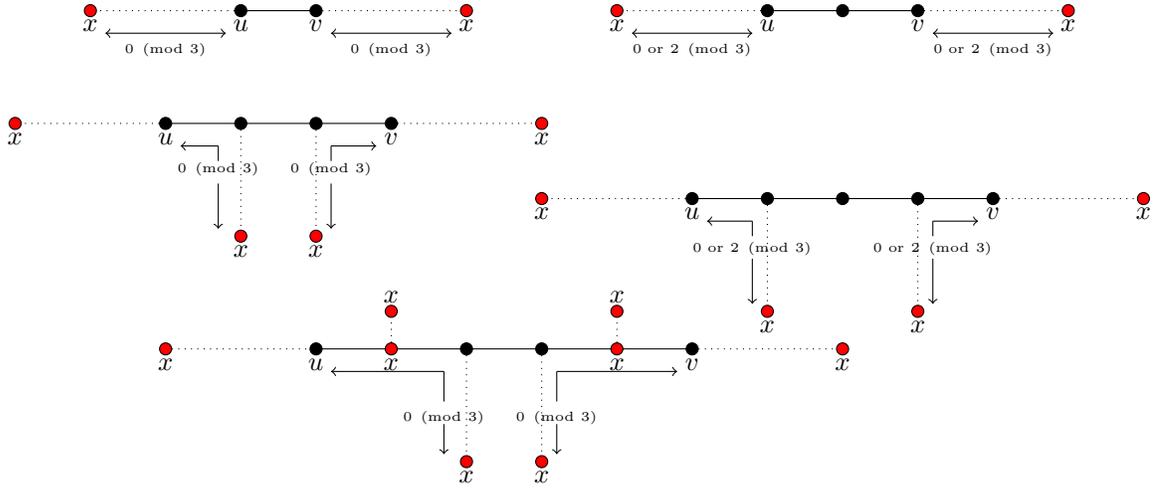
\begin{figure}
\begin{center}
\begin{tikzpicture}

\draw[dotted] (-8,0)--(-10,0); \draw (-8,0)--(-7,0);\draw[dotted] (-7,0)--(-5,0);
\draw[<->] (-9.8,-.3)--(-8.2,-.3); \draw[<->] (-6.8,-.3)--(-5.2,-.3);

\draw  [fill=black](-8,0) circle [radius=0.08]; \draw  [fill=red](-10,0) circle [radius=0.08]; \draw  [fill=black](-7,0) circle [radius=0.08]; \draw  [fill=red](-5,0) circle [radius=0.08]; 

\node [below] at (-8,0) {$u$}; \node [below] at (-7,0) {$v$}; \node [below] at (-5,0) {$x$}; \node [below] at (-10,0) {$x$};

\node [below, font=\tiny] at (-9,-.3) {$0 \pmod 3$}; \node [below, font=\tiny] at (-6,-.3) {$0 \pmod 3$};

\draw[dotted] (-1,0)--(-3,0); \draw (-1,0)--(1,0); \draw[dotted] (1,0)--(3,0); 
\draw[<->] (-2.8,-.3)--(-1.2,-.3); \draw[<->] (1.2,-.3)--(2.8,-.3);

\draw  [fill=black](-1,0) circle [radius=0.08]; \draw  [fill=black](0,0) circle [radius=0.08]; \draw  [fill=black](1,0) circle [radius=0.08]; \draw  [fill=red](-3,0) circle [radius=0.08]; \draw  [fill=red](3,0) circle [radius=0.08];

\node [below] at (-1,0) {$u$}; \node [below] at (1,0) {$v$}; \node [below] at (3,0) {$x$}; \node [below] at (-3,0) {$x$};

\node [below, font=\tiny] at (-2,-.3) {$0$ or $2 \pmod 3$}; \node [below, font=\tiny] at (2,-.3) {$0$ or $2 \pmod 3$};

\draw[dotted] (-9,-1.5)--(-11,-1.5); \draw (-9,-1.5)--(-6,-1.5);\draw[dotted] (-6,-1.5)--(-4,-1.5); \draw[dotted] (-8,-1.5)--(-8,-3);\draw[dotted] (-7,-1.5)--(-7,-3);

\draw[<-] (-8.8,-1.8)--(-8.3,-1.8); \draw (-8.3,-2)--(-8.3, -1.8);\draw[->] (-8.3,-2.3)--(-8.3,-2.9);
\draw[->] (-6.8,-1.8)--(-6.2,-1.8);\draw (-6.8,-1.8)--(-6.8,-2);\draw[->] (-6.8,-2.3)--(-6.8,-2.9);

\draw  [fill=black](-9,-1.5) circle [radius=0.08]; \draw  [fill=black](-8,-1.5) circle [radius=0.08];\draw  [fill=black](-7,-1.5) circle [radius=0.08];
\draw  [fill=black](-6,-1.5) circle [radius=0.08];\draw  [fill=red](-11,-1.5) circle [radius=0.08];\draw  [fill=red](-4,-1.5) circle [radius=0.08];
\draw  [fill=red](-8,-3) circle [radius=0.08];\draw  [fill=red](-7,-3) circle [radius=0.08];

\node [below] at (-9,-1.5) {$u$}; \node [below] at (-6,-1.5) {$v$};\node [below] at (-11,-1.5) {$x$};\node [below] at (-4,-1.5) {$x$}; \node [below] at (-8,-3) {$x$}; \node [below] at (-7,-3) {$x$}; 

\node [below, font=\tiny] at (-8.3,-1.9) {$0 \pmod 3$}; \node [below, font=\tiny] at (-6.8,-1.9) {$0 \pmod 3$};

\draw[dotted] (-2,-2.5)--(-4,-2.5); \draw (-2,-2.5)--(2,-2.5); \draw[dotted] (2,-2.5)--(4,-2.5);\draw[dotted] (-1,-2.5)--(-1,-4);\draw[dotted] (1,-2.5)--(1,-4);

\draw[<-] (-1.8,-2.8)--(-1.2,-2.8); \draw (-1.2,-2.8)--(-1.2, -3);\draw[->] (-1.2,-3.3)--(-1.2,-3.9);
\draw[->] (1.2,-2.8)--(1.8,-2.8);\draw (1.2,-2.8)--(1.2,-3);\draw[->] (1.2,-3.3)--(1.2,-3.9);

\draw  [fill=black](-2,-2.5) circle [radius=0.08]; \draw  [fill=black](-1,-2.5) circle [radius=0.08]; \draw  [fill=black](0,-2.5) circle [radius=0.08]; \draw  [fill=black](1,-2.5) circle [radius=0.08]; \draw  [fill=black](2,-2.5) circle [radius=0.08];
\draw  [fill=red](-4,-2.5) circle [radius=0.08]; \draw  [fill=red](4,-2.5) circle [radius=0.08];\draw  [fill=red](-1,-4) circle [radius=0.08];\draw  [fill=red](1,-4) circle [radius=0.08];

\node [below] at (-2,-2.5) {$u$}; \node [below] at (2,-2.5) {$v$}; \node [below] at (-4,-2.5) {$x$}; \node [below] at (4,-2.5) {$x$};

\node [below] at (-1,-4) {$x$}; \node [below] at (1,-4) {$x$};

\node [below, font=\tiny] at (-1.2,-2.95) {$0$ or $2 \pmod 3$}; \node [below, font=\tiny] at (1.2,-2.95) {$0$ or $2 \pmod 3$};


\draw[dotted] (-9,-4.5)--(-7,-4.5);\draw (-7,-4.5)--(-2,-4.5); \draw[dotted] (-2,-4.5)--(0,-4.5); \draw[dotted] (-5,-4.5)--(-5,-6); \draw[dotted] (-4,-4.5)--(-4,-6);

\draw[<-] (-6.8,-4.8)--(-5.3,-4.8);\draw (-5.3,-4.8)--(-5.3, -5.2);\draw[->] (-5.3,-5.5)--(-5.3,-5.9);

\draw[->] (-3.8,-4.8)--(-2.2,-4.8); \draw (-3.8,-4.8)--(-3.8,-5.2); \draw[->] (-3.8,-5.5)--(-3.8,-5.9);

\draw[dotted] (-6,-4.5)--(-6,-4);
\draw[dotted] (-3,-4.5)--(-3,-4);

\draw [fill=red](-6,-4) circle [radius=0.08]; 
\draw [fill=red](-3,-4) circle [radius=0.08];

\draw  [fill=black](-7,-4.5) circle [radius=0.08];\draw  [fill=red](-6,-4.5) circle [radius=0.08];\draw  [fill=black](-5,-4.5) circle [radius=0.08];
\draw  [fill=black](-4,-4.5) circle [radius=0.08]; \draw  [fill=red](-3,-4.5) circle [radius=0.08]; \draw  [fill=black](-2,-4.5) circle [radius=0.08];
\draw  [fill=red](0,-4.5) circle [radius=0.08]; \draw  [fill=red](-9,-4.5) circle [radius=0.08]; \draw  [fill=red](-5,-6) circle [radius=0.08];
\draw  [fill=red](-4,-6) circle [radius=0.08];

\node [below] at (-7,-4.5) {$u$}; \node [below] at (-2,-4.5) {$v$}; \node [below] at (-9,-4.5) {$x$}; \node [below] at (0,-4.5) {$x$}; \node [below] at (-5,-6) {$x$}; \node [below] at (-4,-6) {$x$};
\node [below] at (-6,-4.5){$x$};
\node [below] at (-3,-4.5){$x$};
\node[above] at (-6,-4){$x$};
\node[above] at (-3,-4){$x$};

\node [below, font=\tiny] at (-3.8,-5.2) {$0 \pmod 3$}; \node [below, font=\tiny] at (-5.3,-5.2) {$0 \pmod 3$};

\end{tikzpicture}
\caption{Resolvability conditions in $T^{3}$ depending on $d_{T}(u,v)$ (all possible positions of $x$ that resolves $u,v$ are depicted by red vertices)}\label{fig1}
\end{center}
\end{figure}
\section{Construction of optimal resolving sets in $T^{3}$}

\noindent In the following, we present the necessary and sufficient conditions for a vertex subset to become a resolving set for cube of trees.

\begin{thm}\label{neccsuff}
Let $T=(V,E)$ be a tree. The necessary and sufficient conditions for a set $S\subset V$
to be a resolving set of $T^{3}$ are

\begin{enumerate}

\item For every edge $uv\in E(T)$, $S$ contains at least one vertex $x$ which is at distance $0$ (mod $3$) from $u$ or $v$.

\item For every edge $uv\in E(T^{2})$, $S$ contains at least one vertex $x$ in $T_{u}$ or $T_{v}$ satisfying min$\{d_{T}(x,u),d_{T}(x,v)\}\equiv \hspace{0.2em} 0 \hspace{0.2em} \mbox{or} \hspace{0.2em}2$  $(\mbox{mod} \hspace{0.2em} 3)$


\item For every edge $uv\in E(T^{3})$, $S$ contains one vertex $x$ either in $T_{u}$ or $T_{v}$ such that $|d_{T}(x,u)-d_{T}(x,v)|=3$, otherwise min $\{d_{T}(x,u),d_{T}(x,v)\}\equiv 0$ (mod $3$).

\item For every pair of four distance vertices $u,v$, $S$ contains one vertex $x$ either in $T_{u}$ or $T_{v}$ such that $|d_{T}(x,u)-d_{T}(x,v)|=4$, otherwise min $\{d_{T}(x,u),d_{T}(x,v)\}\equiv 0$ or $2$ (mod $3$).

\item For every pair of five distance vertices $u,v$,  $S$ contains one vertex $x$ either in $T_{u}$ or $T_{v}$ such that $|d_{T}(x,u)-d_{T}(x,v)|=5 $, otherwise $|d_{T}(x,u)-d_{T}(x,v)|=3$ or min $\{d_{T}(x,u),d_{T}(x,v)\}\equiv 0$ (mod $3$).

\end{enumerate} 
\end{thm}
\begin{proof}
Let $S$ be a resolving set of $T^{3}$ and $x\in S$ resolves a pair of distinct vertices $u,v$. If $x\neq u,v$, then condition $1$, condition $2$, and condition $3$ hold from Lemma \ref{3}. Also, condition $4$ and condition $5$ follow from Corollary \ref{four distance}. By triviality, all the conditions hold if $x=u$ or $v$.

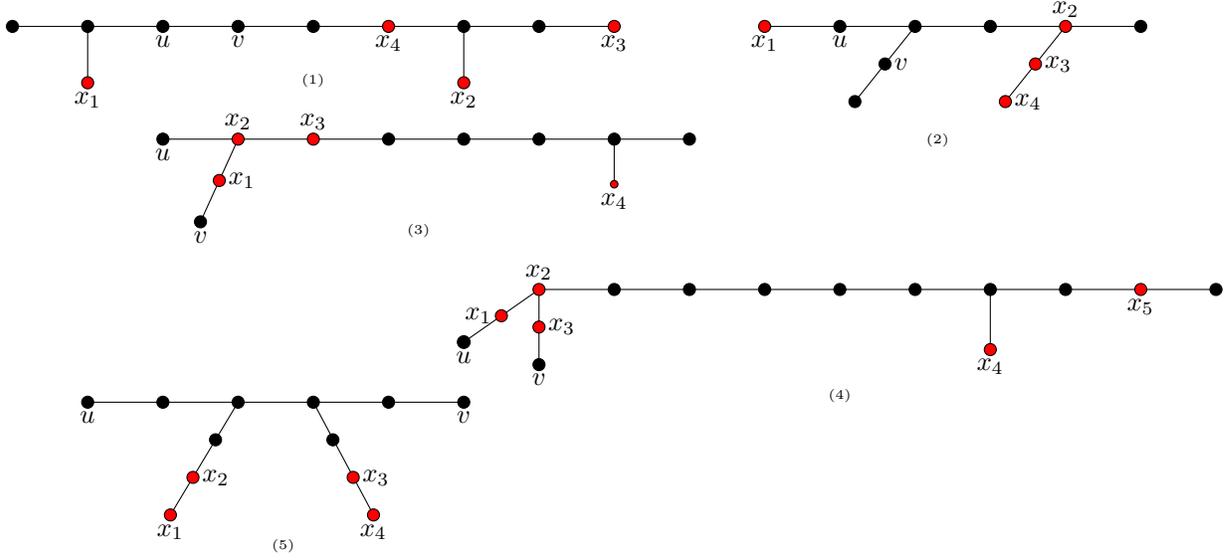
\begin{figure}
\begin{center}
\begin{tikzpicture}

\draw(-13,0)--(-5,0); \draw (-12,0)--(-12,-0.75); \draw (-7,0)--(-7,-0.75);

\draw  [fill=black](-13,0) circle [radius=0.08];\draw  [fill=black](-12,0) circle [radius=0.08];\draw  [fill=black](-11,0) circle [radius=0.08]; \draw  [fill=black](-10,0) circle [radius=0.08]; \draw  [fill=black](-9,0) circle [radius=0.08];\draw  [fill=red](-8,0) circle [radius=0.08]; \draw  [fill=black](-7,0) circle [radius=0.08]; \draw  [fill=black](-6,0) circle [radius=0.08]; \draw  [fill=red](-5,0) circle [radius=0.08]; \draw  [fill=red](-12,-0.75) circle [radius=0.08]; \draw  [fill=red](-7,-0.75) circle [radius=0.08];

\node [below] at (-11,0) {$u$}; \node [below] at (-10,0) {$v$}; \node [below] at (-12,-0.75) {$x_1$}; \node [below] at (-8,0) {$x_4$}; \node [below] at (-5,0) {$x_3$}; \node [below] at (-7,-0.75) {$x_2$};

\node [below, font=\tiny] at (-9,-.5) {$(1)$};


\draw (2,0)--(-3,0);\draw (-1,0)--(-1.8,-1); \draw (1,0)--(.2,-1);

\draw  [fill=red](-3,0) circle [radius=0.08]; \draw  [fill=black](-2,0) circle [radius=0.08];\draw  [fill=black](-1,0) circle [radius=0.08];\draw  [fill=black](0,0) circle [radius=0.08]; \draw  [fill=red](1,0) circle [radius=0.08];  \draw  [fill=black](2,0) circle [radius=0.08];
\draw  [fill=black](-1.8,-1) circle [radius=0.08]; \draw  [fill=black](-1.4,-.5) circle [radius=0.08]; \draw  [fill=red](.2,-1) circle [radius=0.08]; \draw  [fill=red](.6,-.5) circle [radius=0.08];

\node [below] at (-2,0) {$u$}; \node [right] at (-1.4,-.5) {$v$}; \node [below] at (-3,0) {$x_1$}; \node [above] at (1,0) {$x_2$}; \node [right] at (.6,-.5) {$x_3$};\node [right] at (.2,-1) {$x_4$};

\node [below, font=\tiny] at (-.7,-1.3) {$(2)$};


\draw(-4,-1.5)--(-11,-1.5); \draw (-10,-1.5)--(-10.5,-2.6);\draw (-5,-1.5)--(-5,-2.1);

\draw  [fill=black](-11,-1.5) circle [radius=0.08];\draw  [fill=red](-10,-1.5) circle [radius=0.08];  \draw  [fill=red](-9,-1.5) circle [radius=0.08]; \draw  [fill=black](-8,-1.5) circle [radius=0.08];\draw  [fill=black](-7,-1.5) circle [radius=0.08]; \draw  [fill=black](-6,-1.5) circle [radius=0.08];\draw  [fill=black](-5,-1.5) circle [radius=0.08];\draw  [fill=black](-4,-1.5) circle [radius=0.08]; \draw  [fill=black](-10.5,-2.6) circle [radius=0.08];\draw  [fill=red](-10.25,-2.05) circle [radius=0.08];\draw  [fill=red](-5,-2.1) circle [radius=0.05];

\node [below] at (-11,-1.5) {$u$}; \node [below] at (-10.5,-2.6) {$v$};\node [right] at (-10.25,-2.05) {$x_1$};\node [above] at (-10,-1.5) {$x_2$}; \node [above] at (-9,-1.5) {$x_3$}; \node [below] at (-5,-2.1) {$x_4$};

\node [below, font=\tiny] at (-7.6,-2.5) {$(3)$};


\draw (-6,-3.5)--(3,-3.5); \draw (-6,-3.5)--(-7,-4.2); \draw (-6,-3.5)--(-6,-4.5);\draw (0,-3.5)--(0,-4.3);

 \draw  [fill=red](-6,-3.5) circle [radius=0.08];\draw  [fill=black](-5,-3.5) circle [radius=0.08];\draw  [fill=black](-4,-3.5) circle [radius=0.08];\draw  [fill=black](-3,-3.5) circle [radius=0.08];\draw  [fill=black](-2,-3.5) circle [radius=0.08];\draw  [fill=black](-1,-3.5) circle [radius=0.08];\draw  [fill=black](0,-3.5) circle [radius=0.08]; \draw  [fill=black](1,-3.5) circle [radius=0.08]; \draw  [fill=red](2,-3.5) circle [radius=0.08]; \draw  [fill=black](3,-3.5) circle [radius=0.08];

\draw  [fill=black](-7,-4.2) circle [radius=0.085]; \draw  [fill=red](-6.5,-3.85) circle [radius=0.08]; \draw  [fill=black](-6,-4.5) circle [radius=0.08]; \draw  [fill=red](-6,-4) circle [radius=0.08]; \draw  [fill=red](0,-4.3) circle [radius=0.08]; 

\node [below] at (-7,-4.2) {$u$}; \node [below] at (-6,-4.5) {$v$}; \node [left] at (-6.5,-3.85) {$x_1$}; \node [above] at (-6,-3.5) {$x_2$};
\node [right] at (-6,-4) {$x_3$}; \node [below] at (0,-4.3) {$x_4$}; \node [below] at (2,-3.5) {$x_5$};

\node [below, font=\tiny] at (-2,-4.7) {$(4)$};



\draw(-12,-5)--(-7,-5);\draw (-10,-5)--(-10.9,-6.5); \draw (-9,-5)--(-8.2,-6.5); 

\draw [fill=black](-12,-5) circle [radius=0.08]; \draw [fill=black](-11,-5) circle [radius=0.08]; \draw [fill=black](-10,-5) circle [radius=0.08]; \draw [fill=black](-9,-5) circle [radius=0.08]; \draw [fill=black](-8,-5) circle [radius=0.08];  \draw [fill=black](-7,-5) circle [radius=0.08]; 

\draw [fill=red](-10.9,-6.5) circle [radius=0.08]; \draw [fill=red](-10.6,-6) circle [radius=0.08]; \draw [fill=black](-10.3,-5.5) circle [radius=0.08]; 
\draw [fill=red](-8.2,-6.5) circle [radius=0.08];\draw [fill=red](-8.47,-6) circle [radius=0.08]; \draw [fill=black](-8.74,-5.5) circle [radius=0.08]; 

\node [below] at (-12,-5) {$u$}; \node [below] at (-7,-5) {$v$}; \node [below] at (-10.9,-6.5) {$x_1$}; \node [right] at (-10.6,-6) {$x_2$}; \node [right] at (-8.47,-6) {$x_3$}; \node [below] at (-8.2,-6.5) {$x_4$};

\node [below, font=\tiny] at (-9.4,-6.7) {$(5)$};


\end{tikzpicture}
\caption{For the trees $T_{i}$, $1\leq i\leq 5$ and $S_{i}$ (set of all red vertices), all the conditions of Theorem \ref{neccsuff} hold true except condition $i$, which fails for the pair $u,v$ satisfying $d_{T}(u,v)=i$}\label{fig2}
\end{center}
\end{figure}
 
Conversely, let $u,v$ be any two arbitrary vertices of $T^{3}$. We consider the following cases depending on their adjacency in $T$ and prove the existence of a vertex $x\in S$ that resolves $u,v$ in each case \footnote{we omit the trivial case, i.e., when $x=u$ or $v$ from rest of the part of this proof}. (see Figure \ref{fig1} and Figure \ref{fig2})

\vspace{0.7em}
\noindent  \textbf{Case $1$:} Let $u$ and $v$ are adjacent in $T$.  From condition $1$, for each edge $uv$, there exists a vertex (say $x$) from $S$ such $d_{T}(x,u)\equiv 0$ (mod $3$) or $d_{T}(x,v)\equiv 0$ (mod $3$). Without loss of generality we assume $d_{T}(x,u)\equiv 0$ (mod $3$). Then we get min $\{d_{T}(x,u),d_{T}(x,v)\}=d_{T}(x,u)\equiv 0$ ( mod $3$) and hence using Lemma \ref{3}, $x$ resolves $u$ and $v$. 

\vspace{0.7em}
\noindent  \textbf{Case $2$:} Let $u$ and $v$ be nonadjacent in $T$. Consider the following cases according to $d_{T}(u,v)$ is even or odd.

\vspace{0.7em}
\noindent \textbf{Subcase (2a)}: $d_{T}(u,v)$ is even. Let $w$ be the middle vertex of $u-v$ path in $T$ satisfying $d_{T}(u,w)=d_{T}(w,v)$. 

\vspace{0.7em}
 Let $d_{T}(u,v)=2$. Then by condition $2$, one can able to find some $x\in S$ which is at distance $0$ or $2$ (mod $3$) from $uv\in E(T^{2})$. Therefore $x$ resolves $u,v$ follows from Lemma \ref{3}.

\vspace{0.4em}
Let $d_{T}(u,v)=4$. Then, from condition $4$ we get existence of some $x\in S$ such that $|d_{T}(x,u)-d_{T}(x,v)|=4$ or min $\{d_{T}(x,u),d_{T}(x,v)\}\equiv 0$ or $2$ (mod $3$). If $|d_{T}(x,u)-d_{T}(x,v)|=4$, $x$ is either in $T_{u}\setminus\{u\}$ or in $T_{v}\setminus\{v\}$. Hence, by Corollary \ref{four distance}, $x$ resolves $u,v$ in $T^{3}$. For the other case, $x$ must be attached to some branch at $s$ of the $u-v$ path satisfying min $\{d_{T}(x,u),d_{T}(x,v)\}\equiv 0$ or $2$ (mod $3$). Let $d_{T}(x,u)=\mbox{min} \{d_{T}(x,u),d_{T}(x,v)\}$.  Again $d_{T}(x,v)-d_{T}(x,u)=d_{T}(s,v)-d_{T}(s,u)=2$. Therefore, $x$ resolves $u,v$ in $T^{3}$ by Lemma \ref{4}.

\vspace{0.7em}
\noindent Next, we consider the case when $d_{T}(u,v)\geq 6$.
\vspace{0.37em}

Consider two vertices $u_{0},v_{0}$ that occur on either side of $w$ within the path $u-w, w-v$ respectively, satisfying $d_{T}(w,u_{0})=d_{T}(w,v_{0})=1$. Applying condition $2$ for the edge $u_{0}v_{0}\in E(T^{2})$ we get the existence of some $x\in S$. Without loss of generality, we assume $x\in T_{u_{0}}$.

\vspace{0.3em}
If $x$ occurs in the extended path of $u-u_{0}$, then by Lemma \ref{4}, $x$ resolves $u,v$ in $T^{3}$ as $d_{T}(x,v)-d_{T}(x,u)=d_{T}(u,v)\geq 6$.

\vspace{0.3em}

Next, we consider $x$ to be within the $u-u_{0}$ path or attached to some vertex $s$ of the $u-u_{0}$ path. Then $d_{T}(x,u_{0})=\mbox{min}\{d_{T}(x,u_{0}),d_{T}(x,v_{0})\}\equiv 0$ or $2$ (mod $3$). It is easy to note that 
$d_{T}(x,u)=\mbox{min}\{d_{T}(x,u),d_{T}(x,v)\}$.

\vspace{0.6em}

\noindent  a) Let $d_{T}(x,u_{0})\equiv 0 $ (mod $3$).

\vspace{0.7em}
Let $d_{T}(x,u)=d_{T}(x,s)+d_{T}(s,u)=d_{T}(x,u_{0})-d_{T}(s,u_{0})+d_{T}(s,u)=3k-d_{T}(s,u_{0})+d_{T}(s,u)$ for some integer $k\geq 0$. Then $d_{T}(x,v)=d_{T}(x,u_{0})+d_{T}(u_{0},w)+d_{T}(w,v)=3k+1+(d_{T}(u,s)+d_{T}(s,u_{0})+1)$ as $d_{T}(w,v)=d_{T}(u,w)$. Therefore, $d_{T}(x,v)-d_{T}(x,u)=2+2d_{T}(s,u_{0})\geq 4$ when $d_{T}(s,u_{0})\geq 1$. Hence, by Lemma \ref{4}, it follows that $x$ resolves $u,v$. If $d_{T}(s,u_{0})=0$, then also $x$ resolves $u,v$ if $d_{T}(x,u)\equiv 0$ or $2$ (mod $3$) as $d_{T}(x,v)-d_{T}(x,u)=2$.

\vspace{0.7em}
Hence the case remains when $d_{T}(s,u_{0})=0$ (i.e., $s=u_{0}$) and $d_{T}(x,u)\equiv 1$ (mod $3$). Since $d_{T}(x,u_{0})\equiv 0 $ (mod $3$), we have $d_{T}(u,u_{0})=d_{T}(v,v_{0})\equiv 1$ (mod $3$). Therefore, $d_{T}(u,v)=d_{T}(u,u_{0})+d_{T}(u_{0},v_{0})+d_{T}(v_{0},v)\equiv (1+2+1)$ (mod $3$) $\equiv 1$ (mod $3$). Since the distance between $u$ and $v$ is even, $d_{T}(u,v)\geq 10$. 

\vspace{0.7em}
Now consider two vertices $u_{1},v_{1}$ on either side of $w$ satisfying $d_{T}(w,u_{1})=d_{T}(w,v_{1})=2$. Then, applying condition $4$ on the four distance vertices $u_{1}, v_{1}$ we get the existence of some $y\in S$. Two cases may arise here.

\vspace{0.7em}
\noindent i) If $y$ lies in any extended branch of $u_{1}-u$, then by Lemma \ref{4}, it follows that $y$ resolves $u,v$ in $T^{3}$. Again, if $y$ is attached to some vertex $s$ of $u_{1}-u$ path or lies within the $u_{1}-u$ path (i.e., $y=s$) then $d_{T}(y,v)-d_{T}(y,u)=(d_{T}(y,u_{1})+d_{T}(u_{1},v_{1})+d_{T}(v_{1},v))-(d_{T}(y,u_{1})+d_{T}(u,u_{1})-2d_{T}(s,u_{1}))=2d_{T}(s,u_{1})+4\geq 4$ as $d_{T}(u,u_{1})=d_{T}(v,v_{1})$. Hence, applying Lemma \ref{4}, it is easy to conclude that $y$ resolves $u,v$ in $T^{3}$.

\vspace{0.4em}
Similar logic follows if $y$ is attached to some intermediate vertex of the $v_{1}-v$ path or lies within or in the extended path of $v_{1}-v$. 

\vspace{0.7em}
\noindent ii) If $y$ is attached to $u_{0}$ satisfying $d_{T}(y,u_{1})=\mbox{min}\{d_{T}(y,u_{1}),d_{T}(y,v_{1})\}\equiv 0$ or $2$ (mod $3$). Then $d_{T}(y,v)-d_{T}(y,u)=(d_{T}(y,u_{0})+d_{T}(u_{0},v_{1})+d_{T}(v_{1},v))-(d_{T}(y,u_{1})+d_{T}(u_{1},u))=(d_{T}(y,u_{1})-1)+3-d_{T}(y,u_{1})=2$ as $d_{T}(u,u_{1})=d_{T}(v,v_{1})$ and min $\{d_{T}(y,u),d_{T}(y,v)\}=d_{T}(y,u)=d_{T}(y,u_{1})+d_{T}(u_{1},u)\equiv 0$ or $2$ (mod $3$) as $d_{T}(u,u_{1})=d_{T}(u,u_{0})-d_{T}(u_{0},u_{1})\equiv 0$ (mod $3$). Therefore, by Lemma \ref{4} it follows that $y$ resolves $u,v$ in $T^{3}$. 

\vspace{0.4em}
Similarly, one can show that if $y$ is attached to $v_{0}$ satisfying $d_{T}(y,v_{1})=\mbox{min}\{d_{T}(y,v_{1}),d_{T}(y,u_{1})\}\equiv 0$ or $2$ (mod $3$), then $y$ resolves $u,v$.

\vspace{0.7em}
\noindent b) Let $d_{T}(x,u_{0})\equiv 2$ (mod $3$).  Then, analogous to the previous case, one can show that $d_{T}(x,v)-d_{T}(x,u)\geq 4$ if $d_{T}(s,u_{0})\geq 1$. Therefore, $x$ resolves $u,v$ by Lemma \ref{4}. Again, when $d_{T}(s,u_{0})=0$ then $d_{T}(x,v)-d_{T}(x,u)=2$, therefore, if $d_{T}(x,u)\equiv 0$ or $2$ (mod $3$) then $x$ resolves $u,v$.

\vspace{0.7em}
Hence, the only case remains when $s=u_{0}$ and $d_{T}(x,u)\equiv 1$ (mod $3$). Then $d_{T}(u,u_{0})=d_{T}(v,v_{0})\equiv 2$ (mod $3$) as $d_{T}(x,u_{0})\equiv 2$ (mod $3$). Hence $d_{T}(u,v)=d_{T}(u,u_{0})+d_{T}(u_{0},v_{0})+d_{T}(v_{0},v)\equiv 0$ (mod $3$). 

\vspace{0.7em}
Now consider two neighbours of $u_{0}$, one (say $p_{1}$) on the path $u-u_{0}$ and another (say $p_{2}$) on the path $u_{0}-x$. Then $d_{T}(p_{1},p_{2})=d_{T}(p_{1},u_{0})+d_{T}(u_{0},p_{2})=2$.
Clearly, $d_{T}(u,p_{1})\equiv 1$ (mod $3$). Applying condition $2$ on $p_{1}, p_{2}$, we get the existence of some $y\in S$.  Three cases may arise here. 

\vspace{0.7em}
\noindent i) When $y$ occurs in an extended path from $p_{2}$ then min
$\{d_{T}(y,p_{2}),d_{T}(y,p_{1})\}=d_{T}(y,p_{2})\equiv 0$ or $2$ (mod $3$).  Then $d_{T}(y,v)=d_{T}(y,p_{2})+d_{T}(p_{2},u_{0})+d_{T}(u_{0},v_{0})+d_{T}(v_{0},v)$ and $d_{T}(y,u)=d_{T}(y,p_{2})+d_{T}(p_{2},u_{0})+d_{T}(u_{0},u)$. Therefore $d_{T}(y,v)-d_{T}(y,u)=2$ and min $\{d_{T}(y,v),d_{T}(y,u)\}=d_{T}(y,u)=d_{T}(y,p_{2})+d_{T}(p_{2},p_{1})+d_{T}(p_{1},u)\equiv 0$ or $2$ (mod $3$). Hence, $y$ resolves $u,v$ in $T^{3}$ by Lemma \ref{4}.

\vspace{0.7em}
\noindent ii) If $y$ occurs in the intermediate path of $u-u_{0}$, then $d_{T}(y,v)=d_{T}(y,u_{0})+d_{T}(u_{0},v_{0})+d_{T}(v_{0},v)$ and $d_{T}(y,u)=d_{T}(u,u_{0})-d_{T}(y,u_{0})$. Therefore, $d_{T}(y,v)-d_{T}(y,u)=2d_{T}(y,u_{0})+2>3$ clearly. Hence, by Lemma \ref{4}, we can conclude that $y$ resolves $u,v$ in $T^{3}$.

\vspace{0.7em}
\noindent iii) If $y$ occurs in the extended path from $u-u_{0}$ then it also resolves $u,v$ in $T^{3}$ by Lemma \ref{4}.

\vspace{0.7em}
\noindent \textbf{Subcase (2b)}: $d_{T}(u,v)$ is odd. Let $d_{T}(u,v)=2m+1$ for some positive integer $m\geq 1$.

\vspace{0.7em}
When $m=1$, then $d_{T}(u,v)=3$. From condition $3$, either there exists a $x\in S$ such that $|d_{T}(x,v)-d_{T}(x,u)|=3$ (i.e., $x\in T_{u}$ or $T_{v}$) or min$\{d_{T}(x,u),d_{T}(x,v)\}\equiv 0$ (mod $3$) and hence from Lemma \ref{3} the result follows. Again when $m=2$, i.e., $d_{T}(u,v)=5$. Let $(u,u_{0},w_{1},w_{2}, v_{0}, v)$ be the path between $u,v$ in $T$. Then, from  condition $5$, there exists a $x\in S$  either coming from $T_u$ or $T_v$ satisfying $|d_{T}(x,v)-d_{T}(x,u)|=5$, otherwise $|d_{T}(x,v)-d_{T}(x,u)|=3$ or 
 min $\{d_{T}(x,u),d_{T}(x,v)\}=0$ (mod $3$). If $|d_{T}(x,v)-d_{T}(x,u)|=5$ or $3$, then by Lemma \ref{4}, $x$ resolves $u,v$ in $T^{3}$. In the other case, when min$\{d_{T}(x,u),d_{T}(x,v)\}=0$ (mod $3$) and $x$ is in the same component $T_{u,v}$ of $u,v$, it must be attached to the vertex $w_{1}$ or $w_{2}$ satisfying $d_{T}(x,v)-d_{T}(x,u)=1$. Hence, by Lemma \ref{4}, $x$ resolves $u,v$ in $T^{3}$ .

\vspace{0.7em}
Next, we consider the case when $m\geq 3$, i.e., $d_{T}(u,v)\geq 7$. We consider  the $u-v$ path as $(u,\hdots, u_{1},u_{0},w_{1},w_{2},v_{0},\\
v_{1},\hdots, v)$ where $d_{T}(u,w_{1})=d_{T}(v,w_{2})=m, d_{T}(u_{1},v_{1})=5$.

\vspace{0.7em}
\noindent a) Let $m\equiv 0$ (mod $3$).
\vspace{0.3em}

Applying condition $1$ for the edge $w_{1}w_{2}$, we get the existence of a vertex $x\in S$. Without loss of generality, we assume $x\in T_{w_{1}}$. Then min$\{d_{T}(x,w_{1}),d_{T}(x,w_{2})\}=d_{T}(x,w_{1})\equiv 0$ (mod $3$). 
 
\vspace{0.3em}
If $x$ occurs in the extended path of $u-w_{1}$, then $d_{T}(x,v)-d_{T}(x,u)=d_{T}(u,v)\geq 7$. Again, if $x$ occurs in a branch attached to some vertex $s$ within the path $u-w_{1}$, then $d_{T}(x,v)=d_{T}(x,s)+d_{T}(s,w_{1})+d_{T}(w_{1},w_{2})+d_{T}(w_{2},v)$ and $d_{T}(x,u)=d_{T}(x,s)+d_{T}(s,u)=d_{T}(x,s)+d_{T}(u,w_{1})-d_{T}(s,w_{1})$. 
Therefore, $d_{T}(x,v)-d_{T}(x,u)=2d_{T}(s,w_{1})+1\geq 3$ when $d_{T}(s,w_{1})\geq 1$. Again, if $d_{T}(s,w_{1})=0$, i.e., when $s=w_{1}$,  we get $d_{T}(x,u)=d_{T}(x,w_{1})+d_{T}(w_{1},u)\equiv 0$ (mod $3$) and $d_{T}(x,v)-d_{T}(x,u)=1$. Therefore, by Lemma \ref{4}, $x$ resolves $u,v$ in $T^{3}$ for the above cases.

\vspace{0.7em}
\noindent  b) Let $m\equiv 2$ (mod $3$).
\vspace{0.3em}

It is easy to note that $d_{T}(u,u_{1})\equiv 0$ (mod $3$) in this case. Applying condition $5$ to the vertices $u_{1},v_{1}$, we get to know the existence of a $x\in S$. Without loss of generality, we assume $x\in T_{w_{1}}$. Then $d_{T}(x,u_{1})=\text{min}\{d_{T}(x,u_{1}),d_{T}(x,v_{1})\}$ and hence $d_{T}(x,u)=\text{min}\{d_{T}(x,u),d_{T}(x,v)\}$. 


\vspace{0.3em}
When $|d_{T}(x,u_{1})-d_{T}(x,v_{1})|=5$, then $x$ is in the extended path of $u_{1}-w_{1}$. Then $|d_{T}(x,v)-d_{T}(x,u)|=5\geq3$. If $|d_{T}(x,u_{1})-d_{T}(x,v_{1})|=3$, then $x=u_{0}$ or $x$ is on a branch attached to $u_{0}$ as $x\in T_{w_{1}}$. Therefore $d_{T}(x,v)-d_{T}(x,u)=3$. 
Again, if $x$ is attached to $w_{1}$ satisfying $d_{T}(x,u_{1})\equiv 0$ (mod $3$). Then $d_{T}(x,v)= d_{T}(x,w_{1})+d_{T}(w_{1},w_{2})+d_{T}(w_{2},v)$, $d_{T}(x,u)=d_{T}(x,w_{1})+d_{T}(w_{1},u)$ and therefore $d_{T}(x,v)-d_{T}(x,u)=1$. Moreover, $d_{T}(x,u)=d_{T}(x,u_{1})+d_{T}(u_{1},u)\equiv 0$ (mod $3$) in this situation. Hence, by Lemma \ref{4}, $x$ resolves $u,v$ in $T^{3}$.

\vspace{0.6em}
\noindent c) Let $m\equiv 1$ (mod $3$).

\vspace{0.3em}
In this case, $d_{T}(u,u_{0})\equiv 0$ (mod $3$). Since $d_{T}(u_{0},v_{0})=3$, applying condition $3$ on the edge $u_{0}v_{0}$ we get existence of some $x\in S$. Without loss of generality, we assume $x\in T_{w_{1}}$. Therefore $d_{T}(x,u)=\text{min}\{d_{T}(x,u),d_{T}(x,v)\}$. Now if $x$ occurs in the extended path of $u-u_{0}$ or attached to some intermediate vertex of the path $u-u_{0}$, then $|d_{T}(x,v)-d_{T}(x,u)|\geq 3$, therefore by Lemma \ref{4}, $x$ resolves $u,v$ in $T^{3}$.

\vspace{0.3em}
\noindent Therefore, the case remains when $x$ is attached to a branch at $w_{1}$ satisfying min $\{d_{T}(x,u_{0}),d_{T}(x,v_{0})\}=d_{T}(x,u_{0})\equiv 0$ (mod $3$). In this case, we have $d_{T}(x,v)-d_{T}(x,u)= (d_{T}(x,w_{1})+d_{T}(w_{1},w_{2})+d_{T}(w_{2},v))-(d_{T}(x,w_{1})+d_{T}(w_{1},u))=1$ as $d_{T}(w_{2},v)=d_{T}(w_{1},u)$. Furthermore, we get $d_{T}(x,u)=\mbox{min}\{d_{T}(x,u),d_{T}(x,v)\}=d_{T}(x,u_{0})+d_{T}(u_{0},u)\equiv 0$ (mod $3$). Hence, by Lemma \ref{4}, it follows that $x$ resolves $u,v$ in $T^{3}$.  
\vspace{0.7em}

\noindent Thus, we prove that $S$ is a resolving set of $T^{3}$.
\end{proof}

\section{Lower bound for metric dimension of $T^{3}$} 

\noindent In this section, we determine the lower bound of $\beta{(T^{3})}$ for a given tree $T$.

\begin{lem}\label{7}
Let $T=(V,E)$ be a tree, and $v_{0}$ be a major stem of $T$ containing $n_{0}$ legs. Then any metric basis of $T^{3}$ must contain $n_{0}+m_{0}-2$ number of vertices from the legs of $v_{0}$, where $m_{0}\geq 1$ be the number of midlegs attached to $v_{0}$.  
\end{lem}

\begin{proof}
Let $S$ be an arbitrary metric basis of $T^{3}$.  Let $B(v_{0})$ be the set of all leg vertices\footnote{vertices that are along the legs attached to some common major stem} corresponding to $v_{0}$ in $T$ and $n_{0}=p_{0}+m_{0}+l_{0}$ where $p_{0},m_{0},l_{0}$ denote the number of pendants, midlegs, and long legs, respectively.
\vspace{0.4em}

Consider any arbitrary pair of vertices $\{u,v\}\subset B(v_{0})$ satisfying $d_{T}(v_{0},u)=d_{T}(v_{0},v)$. Now if both $u,v$ are on short legs, then $d_{T}(u,v)=2$ or $4$. One can verify from Theorem \ref{neccsuff} that no vertex $w\neq u,v$ can resolve them in $T^{3}$. Therefore, it is necessary to include at least $p_{0}+2m_{0}-2$ vertices in $S$ when $m_{0}\geq 1$. Since $|S|$ is minimum in comparison to any resolving set of $T^{3}$, there will always be a pair of vertices $\{a_{0},b_{0}\}\subset B(v_{0})$ satisfying $d_{T}(v_{0},a_{0})=1, d_{T}(v_{0},b_{0})=2$, left aside from vertex selection while constructing $S$ coming from short legs when $m_{0}\geq 1$. 
  
\vspace{0.4em}
 Consider a long leg $L$ attached to $v_{0}$ and $\{x,y\}\subset B(v_{0})$ be the pair of vertices on $L$ satisfying $d_{T}(v_{0},x)=1$ and $d_{T}(v_{0},y)=2$ respectively.  We consider the pair of vertices $\{a_{0},x\},\{b_{0},y\}$. Clearly, $d_{T}(x,a_{0})=2$ and $d_{T}(y,b_{0})=4$. Now to resolve any of the above pairs and keep $|S|$ to be minimum, it is necessary to include one vertex $z$ from $L$ satisfying $d_{T}(x,z)\equiv 0$ or $2$ (mod $3$) by Theorem \ref{neccsuff}. It can be noted that any $z\neq x,y$ on $L$ satisfying $d_{T}(v_{0},z)\equiv 0$ or $1$ (mod $3$) will work. Since there are $l_{0}$ long legs attached to $v_{0}$, applying similar logic, it is necessary to include $l_{0}$ vertices in $S$ from each of the long legs. Hence, the total number of vertex insertions necessary for constructing any metric basis $S$ of $T^{3}$ is $p_{0}+2m_{0}-2+l_{0}=(p_{0}+m_{0}+l_{0}-1)+(m_{0}-1)=(n_{0}-1)+(m_{0}-1)=n_{0}+m_{0}-2$.
 \end{proof}

 \begin{thm}\label{lowerbound}
Let $T=(V,E)$ be a tree. Then \[\beta(T^{3})\geq\beta(T)+\sum_{i=1,m_{i}\geq 1}^{l}{m_{i}}-l\] where $l$ is the total number of major stems of $T$ containing at least one mid leg and $m_{i}$ denotes the number of mid legs attached to the major stem $v_{i}$, $1\leq i\leq l$.  
\end{thm}

\begin{proof}
Let $S$ be a resolving set of $T^{3}$. Then, by Lemma \ref{1}, it is also a resolving set of $T$. Let $V^{\prime}$ be the set of all major stems of $T$ and each $v_{i}\in V^{\prime}$ contains $n_{i}$ legs, $1\leq i\leq |V^{\prime}|$. From Lemma \ref{7}, we get to know that while constructing any metric basis of $T^{3}$, we necessarily need to insert $n_{i}+m_{i}-2$ number of vertices from the legs of $v_{i}$ where $m_{i}\geq 1$ and the number is $n_{i}-1$ for the remaining major stems (where $m_{i}=0$) from Corollary \ref{30}. Therefore, $|S|\geq \sum\limits_{\substack{i=1}}^{l}(n_{i}+m_{i}-2)+\sum\limits_{\substack{j=1}}^{|V^{\prime}|-l} (n_{j}-1)$. This holds for every resolving set $S$ of $T^{3}$, hence we get $\beta(T^{3})\geq \sum\limits_{\substack{i=1}}^{|V^{\prime}|}(n_{i}-1)+\sum\limits_{\substack{i=1}}^{l} (m_{i}-1)=\beta{(T)}+\sum\limits_{i=1,m_{i}\geq 1}^{l}{m_{i}}-l$ using Corollary \ref{6}.
\end{proof}

\section{Upper Bound for metric dimension of $T^{3}$}  
\noindent In the following theorem, we determine the upper bound of $\beta(T^{3})$ for a given tree $T$.

\begin{thm}\label{upperbound}
Let $T=(V, E)$ be a tree. Then \[\beta{(T^{3})}\leq \beta{(T)}+\sum_{i=1,m_{i}\geq 2}^{l} (m_{i}-1)+M+1-l\] where $M$ is the total number of major stems and $l$ is the number of major stems containing at least two mid legs, and $m_{i}$ denotes the number of mid legs attached to the major stem $v_{i}$ where $1\leq i\leq l$.
\end{thm}

\begin{proof}
\noindent Let $V^{\prime}$ be the set of all major stems 
of $T$ and hence $|V^{\prime}|=M$. Let $p_{v}, m_{v}, l_{v}$ denote the number of pendants, midlegs, and long legs attached to an arbitrary major stem $v\in V^{\prime}$ and $B(v)$ be the set of all leg vertices corresponding to $v$ in $T$. We denote $B[v]=B(v)\cup \{v\}$. Now, depending on the number of different types of legs attached to each major stem, we build a resolving set $S$ for $T^{3}$ in the following way:

\vspace{0.7em}
\noindent \textbf{Construction of $S$:}
\vspace{0.4em}

\noindent $1$) $\mathbf{m_{v}\geq 1}$
\vspace{0.4em}

We choose all the vertices from every midleg in $S$, leaving one midleg aside as unpicked. Now if $l_{v}\geq 1$, we pick the vertex from each long leg, which is at a distance of three from $v$ in $B(v)$. Again, if $p_{v}\geq 1$, then we include all the pendants of $B(v)$ in $S$. 
\vspace{0.4em}

\noindent $2$) $\mathbf{p_{v}\geq 1, m_{v}=0}$
\vspace{0.4em}

Except for one pendant, we choose all the pendants of $B(v)$ in $S$. Also, we include all distance three vertices of $B(v)$ that occur along long legs when $l_{v}\geq 1$.
 
\vspace{0.4em}

\noindent $3$) $\mathbf{p_{v}=m_{v}=0}$
\vspace{0.4em}

It is easy to note that $l_{v}\geq 1$ as $v$ is a major stem. In this case, except for one long leg, we include all vertices that are at a distance of $3$ from $v$ along long legs in $S$.

\noindent As per our above construction, $\beta{(T)}+\sum\limits_{i=1,m_{i}\geq 2}^{l} (m_{i}-1)$ number of vertices has already been included in $S$. We now insert $M+1-l$ extra vertices in $S$. But this insertion of vertices depends on some circumstances listed below.
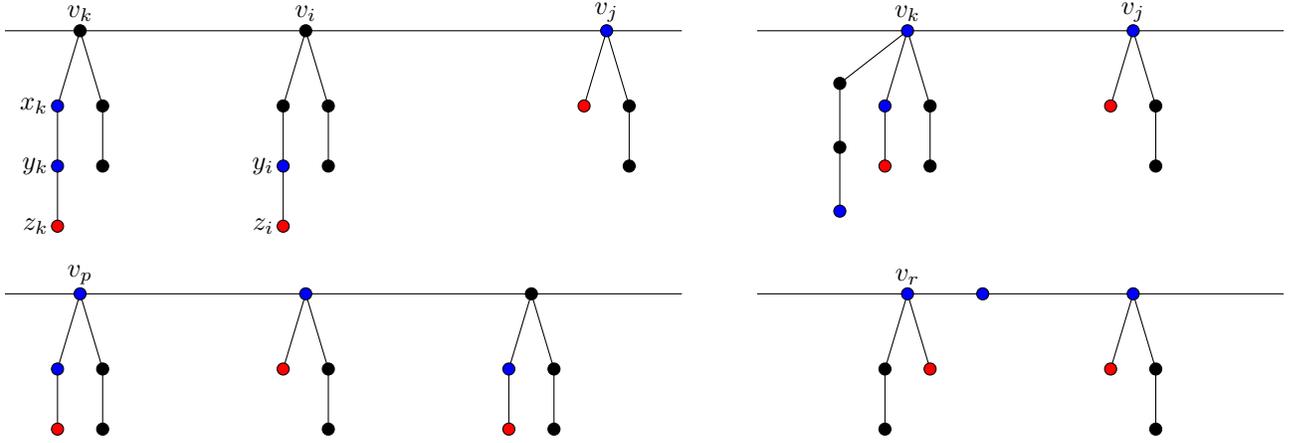
\begin{figure}[ht]\label{bounds1}
\begin{center}
\begin{tikzpicture}

\draw (0,0)--(-1,0);\draw (0,0)--(7,0); 
\draw (0,0)--(-0.3,-1);\draw (0,0)--(0.3,-1);\draw (-0.3,-1)--(-0.3,-1.8);\draw (-0.3,-1.8)--(-0.3,-2.6); \draw (0.3,-1)--(0.3,-1.8);
\draw (3,0)--(2.7,-1);\draw (3,0)--(3.3,-1);\draw (2.7,-1)--(2.7,-1.8);\draw (2.7,-1.8)--(2.7,-2.6); \draw (3.3,-1)--(3.3,-1.8);
\draw (7,0)--(6.7,-1);\draw (7,0)--(7.3,-1);\draw (7.3,-1)--(7.3,-1.8);

\draw (7,0)--(8,0);

\draw  [fill=black](0,0) circle [radius=0.08];

\draw  [fill=red](-0.3,-2.6) circle [radius=0.08];\draw  [fill=red](2.7,-2.6) circle [radius=0.08]; \draw  [fill=black](7.3,-1.8) circle [radius=0.08];

\draw  [fill=blue](-0.3,-1) circle [radius=0.08];\draw  [fill=black](0.3,-1) circle [radius=0.08]; \draw  [fill=blue](-0.3,-1.8) circle [radius=0.08]; \draw  [fill=black](0.3,-1.8) circle [radius=0.08];\draw  [fill=black](3,0) circle [radius=0.08];\draw  [fill=black](2.7,-1) circle [radius=0.08];\draw  [fill=black](3.3,-1) circle [radius=0.08];\draw  [fill=blue](2.7,-1.8) circle [radius=0.08];\draw  [fill=black](3.3,-1.8) circle [radius=0.08];\draw  [fill=blue](7,0) circle [radius=0.08];

\draw  [fill=black](7.3,-1) circle [radius=0.08];\draw  [fill=red](6.7,-1) circle [radius=0.08];

\node [above] at (0,0) {$v_k$};\node [above] at (3,0) {$v_{i}$};\node [above] at (7,0) {$v_{j}$}; \node [left] at (-0.3,-1) {$x_{k}$};\node [left] at (-0.3,-1.8) {$y_{k}$};\node [left] at (-0.3,-2.6) {$z_{k}$};
\node [left] at (2.7,-1.8) {$y_{i}$};\node [left] at (2.7,-2.6) {$z_{i}$};


\draw (9,0)--(16,0);

\draw (11,0)--(10.1,-.7);\draw (10.1,-.7)--(10.1,-2.4);
\draw (11,0)--(10.7,-1);\draw (11,0)--(11.3,-1);\draw (10.7,-1)--(10.7,-1.8); \draw (11.3,-1)--(11.3,-1.8);

\draw (14,0)--(13.7,-1);\draw (14,0)--(14.3,-1);

\draw (14.3,-1)--(14.3,-1.8);


\draw  [fill=blue](11,0) circle [radius=0.08]; \draw  [fill=blue](10.7,-1) circle [radius=0.08];\draw  [fill=black](11.3,-1) circle [radius=0.08]; \draw  [fill=red](10.7,-1.8) circle [radius=0.08]; \draw  [fill=black](11.3,-1.8) circle [radius=0.08];

\draw  [fill=blue](14,0) circle [radius=0.08];\draw  [fill=black](14.3,-1) circle [radius=0.08];\draw  [fill=red](13.7,-1) circle [radius=0.08];\draw  [fill=black](14.3,-1.8) circle [radius=0.08];
\draw  [fill=black](10.1,-.7) circle [radius=0.08];\draw  [fill=blue](10.1,-2.4) circle [radius=0.08];\draw  [fill=black](10.1,-1.55) circle [radius=0.08];

\node [above] at (11,0) {$v_k$};
\node [above] at (14,0) {$v_j$};

\draw (-1,-3.5)--(8,-3.5); 
\draw (0,-3.5)--(-0.3,-4.5);\draw (0,-3.5)--(0.3,-4.5);\draw (-0.3,-4.5)--(-0.3,-5.3); \draw (0.3,-4.5)--(0.3,-5.3);

\draw (3,-3.5)--(2.7,-4.5);\draw (3,-3.5)--(3.3,-4.5); \draw (3.3,-4.5)--(3.3,-5.3);

\draw (6,-3.5)--(5.7,-4.5);\draw (6,-3.5)--(6.3,-4.5);\draw (5.7,-4.5)--(5.7,-5.3);\draw (6.3,-4.5)--(6.3,-5.3);

\draw  [fill=blue](0,-3.5) circle [radius=0.08];

\node [above] at (0,-3.5){$v_{p}$};

\draw  [fill=red](5.7,-5.3) circle [radius=0.08];

\draw  [fill=black](6.3,-5.3) circle [radius=0.08];\draw  [fill=blue](-0.3,-4.5) circle [radius=0.08];

\draw  [fill=black](0.3,-4.5) circle [radius=0.08]; \draw  [fill=red](-0.3,-5.3) circle [radius=0.08]; \draw  [fill=black](0.3,-5.3) circle [radius=0.08];\draw  [fill=blue](3,-3.5) circle [radius=0.08];\draw  [fill=red](2.7,-4.5) circle [radius=0.08];\draw  [fill=black](3.3,-4.5) circle [radius=0.08];\draw  [fill=black](3.3,-5.3) circle [radius=0.08];\draw  [fill=black](6,-3.5) circle [radius=0.08];

\draw  [fill=black](6.3,-4.5) circle [radius=0.08];\draw  [fill=blue](5.7,-4.5) circle [radius=0.08];

\draw (9,-3.5)--(16,-3.5);

\draw (11,-3.5)--(10.7,-4.5);\draw (11,-3.5)--(11.3,-4.5);\draw (10.7,-4.5)--(10.7,-5.3); 

\draw (14,-3.5)--(13.7,-4.5);\draw (14,-3.5)--(14.3,-4.5);

\draw (14.3,-4.5)--(14.3,-5.3);


\draw  [fill=blue](11,-3.5) circle [radius=0.08];\draw  [fill=blue](12,-3.5) circle [radius=0.08]; \draw  [fill=black](10.7,-4.5) circle [radius=0.08];\draw  [fill=red](11.3,-4.5) circle [radius=0.08]; \draw  [fill=black](10.7,-5.3) circle [radius=0.08]; 

\draw  [fill=blue](14,-3.5) circle [radius=0.08];\draw  [fill=black](14.3,-4.5) circle [radius=0.08];\draw  [fill=red](13.7,-4.5) circle [radius=0.08];\draw  [fill=black](14.3,-5.3) circle [radius=0.08];

\node[above] at (11,-3.5) {$v_{r}$};

\end{tikzpicture}
\caption{ Tree $T$ having red vertices as elements of a metric basis of it, blue vertices are extra inserted to form a metric basis $S$ of $T^{3}$, above (left and right) figures correspond to the situation when $T$ contains at least one major stem, and below figures indicate the situation when there is no major stem containing long legs in $T$} \label{fig3i}
\end{center}
\end{figure}

\noindent \textit{Method of insertion of $M+1-l$  extra vertices:}

\vspace{0.6em}

\noindent a) First, we consider the case when there is at least one major stem containing long legs in $T$. (see Figure \ref{fig3i})

\vspace{0.4em}

i) If there is at least one major stem (say $v_{k}$) containing long legs satisfying $m_{v_{k}}\leq 1$, then we select a long leg (say $L_k$) attached to $v_{k}$ from which the vertex $z_{k}$ satisfying $d_{T}(v_{k},z_{k})=3$ already been included in $S$. Next, we pick $x_{k},y_{k}$ from $L_{k}$ satisfying $d_{T}(v_{k},x_{k})=1, d_{T}(v_{k},y_{k})=2$ and include them in $S$. 

\vspace{0.4em}
Now if $v_{i}\neq v_{k}$ be a major stem possessing long legs satisfying $m_{v_{i}}\leq 1$, then we select a long leg $L_{i}$ of $v_{i}$ from where $z_{i}$ is already chosen for $S$ satisfying $d_{T}(v_{i},z_{i})=3$. We pick $y_{i}$ from $L_{i}$ satisfying $d_{T}(v_{i},y_{i})=2$ and insert in $S$.
\vspace{0.4em}
 
Also, we include all those major stems $v_{j}$ in $S$ for which $l_{v_{j}}=0$ and $m_{v_{j}}\leq 1$. 
\vspace{0.4em}

ii) If every major stem that contains at least one long leg also satisfies $m_{v}\geq 2$, then we insert one such major stem (say $v_{k}$) in $S$.  We also insert all those major stems $v_{j}$ in $S$ that satisfy $l_{v_{j}}=0$ and $m_{v_{j}}\leq 1$. 

\vspace{0.7em}

\noindent  b) Next, we consider the case when there is no major stem containing long legs in $T$.
\vspace{0.4em}

If there exists at least one major stem $v_p$ satisfying $m_{v_p}\geq 2$, then include $v_p$ in $S$, otherwise, we include a neighbour of an arbitrary major stem $v_{r}$ in $S$, which does not belong to $B(v_{r})$. We also include all those major stems in $S$ which contain at most one midleg attached to them.

\vspace{0.7em}
\noindent Therefore, the maximum number of extra vertex insertions in the aforementioned scenarios are $M+1-l$.
\vspace{0.6em}

\noindent \textbf{\textit{proof showing that $S$ is a resolving set of $T^{3}$: }} 

\vspace{0.3em}
\noindent We now show that $S$ resolves every pair of vertices $u,v\in V\setminus S$. For this, it is sufficient to prove for the cases when $d_{T}(u,v)\leq 5$ as per Theorem \ref{neccsuff}. Recall that, in $T$, there always exists a unique path joining any two vertices. From the construction of $S$, one can observe that there always exists a major stem $v_{1}$ (say) ($v_{1}=v_{k}\hspace{0.4em}\mbox{or} \hspace{0.5em} v_{p} \hspace{0.4em}\mbox{or} \hspace{0.5em} v_{r}$ in Figure \ref{fig3i}) having three consecutive vertices of $B[v_{1}]$ (or two vertices from $B[v_{1}]$ and one is the neighbour of $v_{1}$ that does not belong to $B[v_{1}]$) and all other major stems having two consecutive vertices from their legs included in $S$ that occur in the extended path of $u-v$ (i.e., in $T_{u}$ or $T_{v}$) or within the same component of $u,v$ (i.e., in $T_{u,v}$), then using Corollary \ref{distance} and Corollary \ref{four distance}, $u,v$ can be resolved by one of these leg vertices that has been selected for $S$. 
\end{proof}

\begin{figure}\label{bounds}
\begin{center}
\begin{tikzpicture}

\draw[<-][dotted] (0,0.5)--(2,0.5);
\draw[->][dotted] (2,0.5)--(3,0.5);
\node at (1.4,0.8) {\tiny{0 (mod 3)}};

\draw[<-][dotted] (3.1,0.5)--(6,0.5);
\draw[->][dotted] (6,0.5)--(7,0.5);
\node at (5.4,0.8){\tiny{1 (mod 3)}};
 
\node [right] at (3.3,-1.2) {$x_{0}$};
\node[above] at (6,2.8) {$u_{1}$};
\node[above] at (6.5,0.8) {$u_{2}$};
\node[below] at (5.3,-2.8) {$u_{k}$};

\draw (0,0)--(7,0); \draw (9,0)--(16,0);
\draw (0,0)--(-0.3,-1);\draw (0,0)--(0.3,-1);\draw (-0.3,-1)--(-0.3,-1.8); \draw (0.3,-1)--(0.3,-1.8);
\draw (3,0)--(2.7,-1);\draw (3,0)--(3.3,-1);\draw (2.7,-1)--(2.7,-1.8); \draw (3.3,-1)--(3.3,-1.8);
\draw (7,0)--(6.7,-1);\draw (7,0)--(7.3,-1);
\draw (6,0)--(6.4,.8);
\draw (3,0)--(6.5,2.3);\draw (6.5,2.3)--(6.2,1.5);\draw (6.5,2.3)--(6.8,1.5);\draw (5.8,1.85)--(5.9,2.7);
\draw (3,0)--(6.5,-2.3);\draw (6.5,-2.3)--(6.2,-3.1);\draw (6.5,-2.3)--(6.8,-3.1);\draw (5.8,-1.85)--(5.3,-2.7);

\draw  [fill=black](0,0) circle [radius=0.08];\draw  [fill=black](1,0) circle [radius=0.08];\draw  [fill=black](2,0) circle [radius=0.08];\draw  [fill=black](5,0) circle [radius=0.08];\draw  [fill=black](4,0) circle [radius=0.08]; \draw  [fill=blue](-0.3,-1) circle [radius=0.08];\draw  [fill=black](0.3,-1) circle [radius=0.08]; \draw  [fill=red](-0.3,-1.8) circle [radius=0.08]; \draw  [fill=black](0.3,-1.8) circle [radius=0.08];\draw  [fill=blue](3,0) circle [radius=0.08];\draw  [fill=blue](2.7,-1) circle [radius=0.08];\draw  [fill=black](3.3,-1) circle [radius=0.08];\draw  [fill=red](2.7,-1.8) circle [radius=0.08];\draw  [fill=black](3.3,-1.8) circle [radius=0.08];\draw  [fill=blue](7,0) circle [radius=0.08];\draw  [fill=black](7.3,-1) circle [radius=0.08];\draw  [fill=red](6.7,-1) circle [radius=0.08];\draw  [fill=black](6,0) circle [radius=0.08];\draw  [fill=black](6.4,.8) circle [radius=0.08];\draw  [fill=blue](6.5,2.3) circle [radius=0.08];\draw  [fill=red](6.2,1.5) circle [radius=0.08];\draw  [fill=black](6.8,1.5) circle [radius=0.08];
\draw  [fill=black](5.8,1.85) circle [radius=0.08];\draw  [fill=black](5.9,2.7) circle [radius=0.08];\draw  [fill=black](4,0.66) circle [radius=0.08];\draw  [fill=black](5,1.32) circle [radius=0.08]; \draw  [fill=black](5.8,-1.85) circle [radius=0.08];
\draw  [fill=black](5.3,-2.7) circle [radius=0.08]; \draw  [fill=black](4,-0.66) circle [radius=0.08];\draw  [fill=black](5,-1.32) circle [radius=0.08];
\draw  [fill=blue](6.5,-2.3) circle [radius=0.08];\draw  [fill=red](6.2,-3.1) circle [radius=0.08];\draw  [fill=black](6.8,-3.1) circle [radius=0.08];

\draw [dashed](6,-.5)--(6,-1.3);

\draw (9,0)--(8.7,-1);\draw (9,0)--(9.3,-1);\draw (8.7,-1)--(8.7,-1.8); \draw (9.3,-1)--(9.3,-1.8);
\draw (12,0)--(11.7,-1);\draw (12,0)--(12.3,-1);\draw (11.7,-1)--(11.7,-1.8); \draw (12.3,-1)--(12.3,-1.8);
\draw (16,0)--(15.7,-1);\draw (16,0)--(16.3,-1);
\draw (15,0)--(15.4,.8);
\draw (12,0)--(15,2.3);\draw (15,2.3)--(14.7,1.5);\draw (15,2.3)--(15.3,1.5);\draw (15,2.3)--(15.9,1.9);\draw (14,1.53)--(14.2,2.5);
\draw (12,0)--(15,-2.3);\draw (15,-2.3)--(14.7,-3.1);\draw (15,-2.3)--(15.3,-3.1);\draw (14,-1.53)--(13.7,-2.5);

\draw  [fill=black](9,0) circle [radius=0.08]; \draw  [fill=blue](8.7,-1) circle [radius=0.08];\draw  [fill=black](9.3,-1) circle [radius=0.08]; \draw  [fill=red](8.7,-1.8) circle [radius=0.08]; \draw  [fill=black](9.3,-1.8) circle [radius=0.08];\draw  [fill=black](10,0) circle [radius=0.08];\draw  [fill=black](11,0) circle [radius=0.08];\draw  [fill=blue](12,0) circle [radius=0.08];\draw  [fill=black](12.3,-1) circle [radius=0.08];\draw  [fill=blue](11.7,-1) circle [radius=0.08];\draw  [fill=red](11.7,-1.8) circle [radius=0.08];\draw  [fill=black](12.3,-1.8) circle [radius=0.08];\draw  [fill=black](13,0) circle [radius=0.08];\draw  [fill=black](14,0) circle [radius=0.08];\draw  [fill=black](15,0) circle [radius=0.08];\draw  [fill=black](15.4,.8) circle [radius=0.08];
\draw  [fill=blue](16,0) circle [radius=0.08]; \draw  [fill=black](16.3,-1) circle [radius=0.08];\draw  [fill=red](15.7,-1) circle [radius=0.08];

\draw  [fill=blue](15,2.3) circle [radius=0.08];\draw  [fill=red](14.7,1.5) circle [radius=0.08];\draw  [fill=red](15.3,1.5) circle [radius=0.08];\draw  [fill=black](14.2,2.5) circle [radius=0.08];\draw  [fill=black](15.9,1.9) circle [radius=0.08];\draw  [fill=black](14,1.5) circle [radius=0.08]; \draw  [fill=black](13.3,1) circle [radius=0.08];\draw  [fill=black](12.6,.46) circle [radius=0.08];

\draw  [fill=blue](15,-2.3) circle [radius=0.08];\draw  [fill=red](14.7,-3.1) circle [radius=0.08];\draw  [fill=black](15.3,-3.1) circle [radius=0.08];\draw  [fill=black](14,-1.5) circle [radius=0.08]; \draw  [fill=black](13.7,-2.5) circle [radius=0.08]; \draw  [fill=black](13.3,-1) circle [radius=0.08];\draw  [fill=black](12.6,-.46) circle [radius=0.08];

\draw [dashed](15,-.5)--(15,-1.3);

\node [above] at (0,0) {$v_1$};\node [above] at (3,0) {$v_{0}$};\node [right] at (6.5,2.3) {$w_{1}$};\node [right] at (7,0) {$w_{2}$};\node [right] at (6.5,-2.3) {$w_{k}$};
\node [above] at (9,0) {$v_1$};\node [above] at (12,0) {$v_{0}$};\node [above] at (15,2.3) {$w_{1}$};\node [right] at (16,0) {$w_{2}$};\node [right] at (15,-2.3) {$w_{k}$};

\draw[<-][dotted] (9,0.5)--(11,0.5);
\draw[->][dotted] (11,0.5)--(12,0.5);
\node at (10.4,0.8) {\tiny{0 (mod 3)}};

\draw[<-][dotted] (12.1,0.5)--(13,0.5);
\draw[->][dotted] (13,0.5)--(16,0.5);
\node at (14.3,0.8){\tiny{1 (mod 3)}};
 \node [right] at (12.3,-1.2) {$x_{0}$};

\node[above] at (14.1,2.6) {$u_{1}$};
\node[above] at (15.5,0.8) {$u_{2}$};
\node[below] at (13.5,-2.6) {$u_{k}$};

\end{tikzpicture}
\caption{Tree $T$ with $\beta{(T^{3})}=n$ (left when $n$ is odd, right when $n$ is even) where the red vertices form the metric basis of $T$ and the blue vertices are extra inserted to form a metric basis of $T^{3}$ } \label{fig3}
\end{center}

\end{figure}
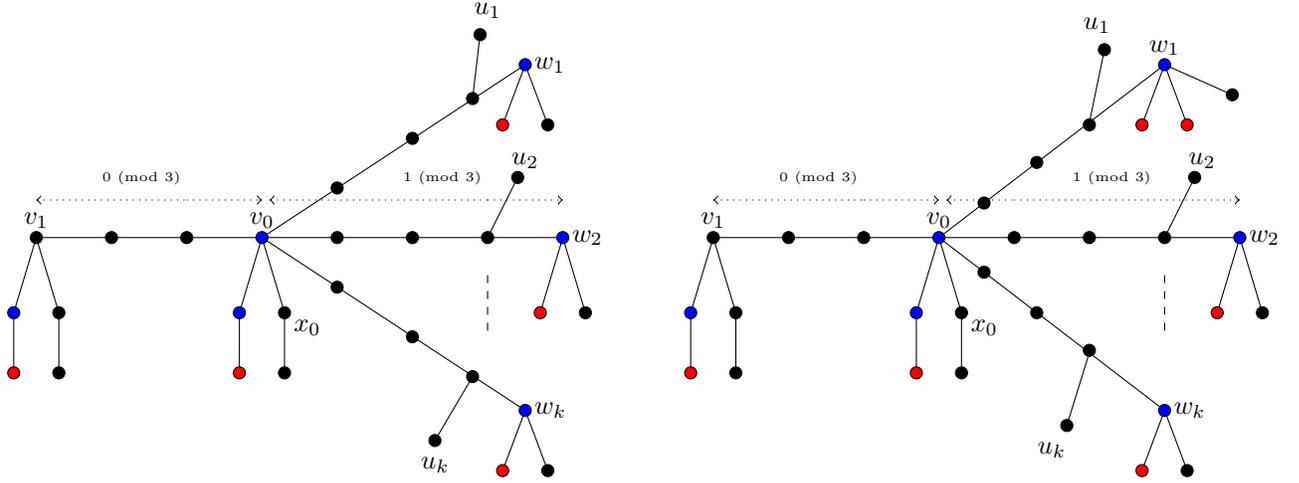

\begin{lem}
For any positive integer $n$ there always exists a tree $T$ satisfying $\beta{(T^{3})}=n$.  
\end{lem}

\begin{proof}
First, we consider the situation when $n$ is even. For this, we consider a tree $T$ (see Figure \ref{fig3}) having $M=\dfrac{n}{2}-1$ number of major stems. Here each of the two major stems $v_{0},v_{1}$ contains exactly two midlegs satisfying $d_{T}(v_{0},v_{1})\equiv 0$ (mod $3$), other $k=M-2=\dfrac{n}{2}-3$ major stems $w_{i}, 1\leq i\leq k$ contain pendants as their only legs where $d_{T}(v_{0},w_{i})\equiv 1$ (mod $3$). Furthermore, we consider $k-1$ of these major stems to contain exactly two pendants and one among them to contain exactly three pendants. From Theorem \ref{0}, it is clear that the metric dimension of $T$, i.e., $\beta{(T)}=\dfrac{n}{2}$. 

\vspace{0.4em}
Below, we construct a resolving set $S$ of $T^{3}$. Since $T$ contains exactly two major stems having two midlegs, from Theorem \ref{lowerbound} it follows that we need to insert at least two more vertices from these midlegs in $S$. Also, we need to include $k$ $(=\dfrac{n}{2}-3)$ more vertices in $S$ to resolve the following pair of vertices $\{u_{1},w_{1}\}$, $\{u_{2}, w_{2}\},\hdots,\{u_{k}w_{k}\}$ in $T^{3}$. We insert $w_{1},w_{2},\hdots, w_{k}$ in $S$. Again, no vertex of $S$ inserted so far can resolve the vertices $v_{0},x_{0}$ in $T^{3}$, therefore, we include one more vertex $v_{0}$ in $S$. One can verify that by applying Theorem \ref{neccsuff}, $S$ becomes a resolving set of $T^{3}$. Furthermore, $|S|\geq \beta{(T)}+2+k+1=\dfrac{n}{2}+3+\dfrac{n}{2}-3=n$.
From Theorem \ref{upperbound} it follows that $\beta{(T^{3})}\leq \beta{(T)}+\sum\limits_{i=1,m_{i}\geq 2}^{l} (m_{i}-1)+M+1-l=\dfrac{n}{2}+2+ (\dfrac{n}{2}-1)+1-2=n$. Therefore, $\beta{(T^{3})}=n$ and hence $S$ becomes a metric basis of $T^{3}$.

\vspace{0.4em}
\noindent Next, we consider the case when $n$ is odd. Then we consider a tree $T$ (see Figure \ref{fig3}) having $M=\dfrac{n-1}{2}$ number of major stems, where each of the two major stems $v_{0},v_{1}$ contains exactly two midlegs satisfying $d_{T}(v_{0},v_{1})\equiv 0$ (mod $3$) and other $k=M-2=\dfrac{n-5}{2}$ number of major stems $w_{i}, 1\leq i\leq k$ contain two pendants each satisfying $d_{T}(v_{0},w_{i})\equiv 1$ (mod $3$). Proceeding similarly as above, one can verify that $\beta{(T)}=\dfrac{n-1}{2}$ and a minimum resolving set $S$ of $T^{3}$ contains exactly $n$ vertices, hence $\beta{(T^{3})}=n$.
\end{proof}

The following corollary is immediate from the above lemma:

\begin{cor}
Given the lower and upper bounds of $\beta{(T^{3})}$ for a tree $T$, there always exist trees attaining every value between the bounds.
\end{cor}

\section{Metric dimension of some well-known cube of trees}
In this section, we present some well-known cubes of trees (e.g., caterpillar, lobster tree, spider tree, and $d$-regular tree) that have attained the expected bounds for the metric dimension.

Let $P$ be the central path \footnote{longest path between any two pendant vertices of a tree} of caterpillar/lobster, and $v_0$, $v_n$ be the starting and ending major stems on $P$. The total number of major stems of any of the trees above-mentioned containing at least two midlegs is denoted by $l$. On the other hand, $m_{i}$ denotes the number of midlegs attached to the major stem $v_{i}$, where $1\leq i\leq l$. Below, we construct the resolving sets $S_0$ and $S$ of $T$ and $T^3$  respectively. In each of the figures in this section, the red vertices form $S_{0}$. One can verify that such choices can be made by Corollary \ref{6}. Furthermore, $S$ can be obtained by inserting the blue vertices in $S_{0}$.
Following Theorem \ref{neccsuff}, it can be verified that $S$ resolves any two arbitrary vertices of $V$. One can find the lower and upper bounds of $\beta(T^{3})$ by applying Theorem \ref{lowerbound} and Theorem \ref{upperbound} respectively.

\begin{exmp}
Let $T=(V, E)$ be a caterpillar. It is easy to observe that there can not be any midleg (or long leg) attached to any stem except $v_{0}$ or $v_{n}$. Furthermore, if there is any midleg or long leg attached to $v_{0}$ or $v_{n}$, then that should be one in number. Also, no long leg and midleg can occur simultaneously at $v_{0}$ or $v_{n}$. Again, while constructing $S$, first we consider that $v_{0}$ contains a long leg (or mid leg) attached to it.  A similar choice of vertices can be made for $S$ if $v_{n}$ contains a long leg (or mid leg) attached to $v_{n}$ and $v_{0}$ contains only pendants. Another case remains when $v_{0},v_{n}$ contains only pendants attached to them. Hence, $\beta{(T)}\leq \beta{(T^{3})}\leq \beta{(T)}+3$. (see Figure \ref{caterpillar})  
\end{exmp}

\begin{figure}
\begin{center}
\begin{tikzpicture}
\node at (0,0.3){$v_{0}$};
\node at (7,0.3){$v_{n}$};

\draw (-1.1,-1.3)--(-.8,-.3);\draw (-.8,-.3)--(0,0);\draw (0,0)--(7,0); \draw (0,0)--(0,-1);\draw (0,0)--(0.5,-1); \draw (1,0)--(1.5,-.8); \draw (3,0)--(3.5,-.8); \draw (3,0)--(3,-.8);\draw (3,0)--(2.5,-.8);\draw (4,0)--(4.5,-.8);\draw (4,0)--(3.8,-.8);\draw (5,0)--(5,-.8);\draw (7,0)--(6.5,-.8);\draw (7,0)--(7.5,-.8);\draw (7,0)--(7.9,-.3);\draw (7.9,-.3)--(8,-.8);

\draw  [fill=blue](0,0) circle [radius=0.08]; \draw  [fill=black](1,0) circle [radius=0.08];\draw  [fill=black](2,0) circle [radius=0.08]; \draw  [fill=black](3,0) circle [radius=0.08]; \draw  [fill=black](4,0) circle [radius=0.08]; \draw  [fill=black](5,0) circle [radius=0.08]; \draw  [fill=black](6,0) circle [radius=0.08];\draw  [fill=blue](7,0) circle [radius=0.08];\draw  [fill=blue](-0.95,-.8) circle [radius=0.08]; \draw  [fill=black](-1.1,-1.3) circle [radius=0.08]; \draw  [fill=red](-.8,-.3) circle [radius=0.08]; \draw  [fill=red](0,-1) circle [radius=0.08]; \draw  [fill=black](0.5,-1) circle [radius=0.08]; \draw  [fill=black](1.5,-.8) circle [radius=0.08]; \draw  [fill=black](3.5,-.8) circle [radius=0.08];\draw  [fill=red](3,-.8) circle [radius=0.08]; \draw  [fill=red](2.5,-.8) circle [radius=0.08];\draw  [fill=black](4.5,-.8) circle [radius=0.08]; \draw  [fill=red](3.8,-.8) circle [radius=0.08];\draw  [fill=black](5,-.8) circle [radius=0.08];\draw  [fill=red](6.5,-.8) circle [radius=0.08];\draw  [fill=red](7.5,-.8) circle [radius=0.08];\draw  [fill=black](7.9,-.3) circle [radius=0.08];\draw  [fill=black](8,-.8) circle [radius=0.08];


\node at (0,-1.7){$v_{0}$};
\node at (7,-1.7){$v_{n}$};

\draw (-.8,-2.3)--(0,-2);\draw (0,-2.6)--(0,-2);\draw (.5,-2.5)--(0,-2);\draw (0,-2)--(7,-2); \draw (2,-2)--(1.8,-2.5);\draw (2,-2)--(2.4,-2.48); \draw (3,-2)--(3,-2.5); \draw (5,-2)--(5,-2.5); \draw (5,-2)--(4.65,-2.5);\draw (5,-2)--(5.4,-2.48);\draw (6,-2)--(6,-2.5); \draw (7,-2)--(7,-2.5); \draw (7,-2)--(6.65,-2.5);\draw (7,-2)--(7.4,-2.48);

\draw  [fill=red](-.8,-2.3) circle [radius=0.08]; \draw  [fill=red](0,-2.6) circle [radius=0.08]; \draw  [fill=black](.5,-2.5) circle [radius=0.08]; \draw  [fill=blue](0,-2) circle [radius=0.08]; \draw  [fill=blue](1,-2) circle [radius=0.08]; \draw  [fill=black](2,-2) circle [radius=0.08]; \draw  [fill=black](3,-2) circle [radius=0.08];\draw  [fill=black](4,-2) circle [radius=0.08];\draw  [fill=black](5,-2) circle [radius=0.08];\draw  [fill=black](6,-2) circle [radius=0.08];\draw  [fill=blue](7,-2) circle [radius=0.08];
\draw  [fill=red](1.8,-2.5) circle [radius=0.08];  \draw  [fill=black](2.4,-2.48) circle [radius=0.08];  \draw  [fill=black](3,-2.5) circle [radius=0.08];  \draw  [fill=red](5,-2.5) circle [radius=0.08]; \draw  [fill=red](4.65,-2.5) circle [radius=0.08];\draw  [fill=black](5.4,-2.48) circle [radius=0.08];\draw  [fill=black](6,-2.5) circle [radius=0.08];\draw  [fill=red](7,-2.5) circle [radius=0.08];\draw  [fill=red](6.65,-2.5) circle [radius=0.08];\draw  [fill=black](7.4,-2.48) circle [radius=0.08];
\end{tikzpicture}
\end{center}
\caption{caterpillar}\label{caterpillar}
\end{figure}
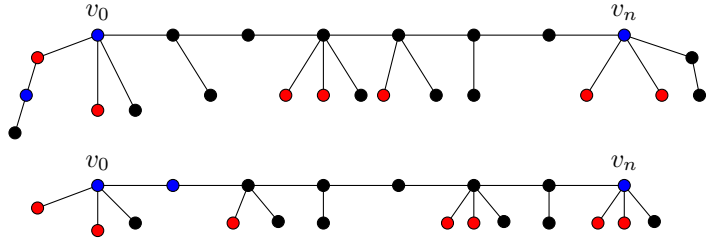

\begin{figure}
\begin{center}
\begin{tikzpicture}
\draw (-.6,-3.8)--(-.6,-4.8);\draw (-.6,-3.8)--(0,-3.5);\draw (-.2,-4)--(0,-3.5);\draw (-.2,-4)--(-.2,-4.5); \draw (.2,-4)--(0,-3.5);\draw (.2,-4)--(.2,-4.5);

\draw (.5,-3.8)--(0,-3.5);\draw (0,-3.5)--(7,-3.5); \draw (1,-3.5)--(1,-4);\draw (3,-3.5)--(2.8,-4); \draw (3,-3.5)--(3.2,-4); \draw (3.2,-4)--(3.2,-4.5);
\draw (3.8,-4)--(4,-3.5);\draw (3.8,-4)--(3.8,-4.5); \draw (4.2,-4)--(4,-3.5);\draw (4.2,-4)--(4.2,-4.5);   \draw (6,-3.5)--(6,-4);
\draw (7.5,-3.8)--(7,-3.5);\draw (7.5,-3.8)--(7.5,-4.5); \draw (6.8,-4)--(7,-3.5); \draw (7.2,-4)--(7,-3.5); \draw (6,-4)--(6.3,-4.5);\draw (6,-4)--(5.7,-4.5);

\draw  [fill=black](-.6,-4.8) circle [radius=0.08]; \draw  [fill=red](-.6,-3.8) circle [radius=0.08]; \draw  [fill=blue](-.6,-4.3) circle [radius=0.08];\draw  [fill=blue](0,-3.5) circle [radius=0.08]; \draw  [fill=blue](-.2,-4) circle [radius=0.08]; \draw  [fill=red](-.2,-4.5) circle [radius=0.08]; \draw  [fill=black](.2,-4) circle [radius=0.08];\draw  [fill=black](.2,-4.5) circle [radius=0.08];\draw  [fill=red](.5,-3.8) circle [radius=0.08]; \draw  [fill=black](-.6,-4.8) circle [radius=0.08]; \draw  [fill=black](1,-3.5) circle [radius=0.08]; \draw  [fill=black](2,-3.5) circle [radius=0.08]; \draw  [fill=black](3,-3.5) circle [radius=0.08];\draw  [fill=black](4,-3.5) circle [radius=0.08]; \draw  [fill=black](5,-3.5) circle [radius=0.08]; \draw  [fill=black](6,-3.5) circle [radius=0.08];\draw  [fill=blue](7,-3.5) circle [radius=0.08];     \draw  [fill=red](2.8,-4) circle [radius=0.08];
\draw  [fill=black](3.2,-4) circle [radius=0.08];\draw  [fill=black](3.2,-4.5) circle [radius=0.08];\draw  [fill=black](4.2,-4) circle [radius=0.08];\draw  [fill=black](4.2,-4.5) circle [radius=0.08];

\draw  [fill=black](6,-4) circle [radius=0.08];

\draw  [fill=black](7.5,-3.8) circle [radius=0.08]; \draw  [fill=black](7.5,-4.5) circle [radius=0.08];\draw  [fill=red](6.8,-4) circle [radius=0.08]; \draw  [fill=red](7.2,-4) circle [radius=0.08];  \draw  [fill=red](3.8,-4.5) circle [radius=0.08];

\draw  [fill=blue](3.8,-4) circle [radius=0.08]; 

\draw  [fill=black](1,-4) circle [radius=0.08]; \draw  [fill=black](6.3,-4.5) circle [radius=0.08]; \draw  [fill=red](5.7,-4.5) circle [radius=0.08];

\node [above] at (0,-3.5){$v_{0}$};
\node [above] at (7,-3.5){$v_{n}$};


\draw (.6,-6.4)--(1,-6);\draw (.6,-6.4)--(.6,-7.5); \draw (.9,-6.8)--(1,-6); \draw (1.2,-6.8)--(1,-6);\draw (1.5,-6.4)--(1,-6);\draw (1.5,-6.4)--(1.6,-7);
\draw (3.3,-6.5)--(3,-6);\draw (3,-6)--(3,-6.5);\draw (2.7,-6.5)--(3,-6);\draw (4,-6)--(4,-6.5);\draw (5.3,-6.5)--(5,-6);\draw (5.3,-6.5)--(5.3,-7);\draw (4.7,-6.5)--(5,-6);\draw (4.7,-6.5)--(4.7,-7);\draw (7.4,-6.5)--(7,-6);\draw (7.4,-6.5)--(7.4,-7.5);\draw (6.6,-6.5)--(7,-6);\draw (6.6,-6.5)--(6.6,-7);\draw (7,-6)--(7,-6.5);   \draw (1,-6)--(7,-6);

\draw  [fill=blue](1,-6) circle [radius=0.08]; \draw  [fill=red](.6,-6.4) circle [radius=0.08]; \draw  [fill=blue](.6,-6.95) circle [radius=0.08]; \draw  [fill=black](.6,-7.5) circle [radius=0.08];\draw  [fill=red](.9,-6.8) circle [radius=0.08]; \draw  [fill=red](1.2,-6.8) circle [radius=0.08]; \draw  [fill=red](1.2,-6.8) circle [radius=0.08];\draw  [fill=black](1.5,-6.4) circle [radius=0.08]; \draw  [fill=blue](1,-6) circle [radius=0.08]; \draw  [fill=black](1.6,-7) circle [radius=0.08];\draw  [fill=red](3.3,-6.5) circle [radius=0.08]; \draw  [fill=black](3,-6) circle [radius=0.08]; \draw  [fill=black](3,-6.5) circle [radius=0.08]; \draw  [fill=red](2.7,-6.5) circle [radius=0.08];\draw  [fill=black](4,-6) circle [radius=0.08]; \draw  [fill=black](4,-6.5) circle [radius=0.08]; \draw  [fill=black](5.3,-6.5) circle [radius=0.08];\draw  [fill=black](5,-6) circle [radius=0.08]; \draw  [fill=black](5.3,-7) circle [radius=0.08]; \draw  [fill=blue](4.7,-6.5) circle [radius=0.08];\draw  [fill=red](4.7,-7) circle [radius=0.08]; \draw  [fill=black](7,-6) circle [radius=0.08];\draw  [fill=red](7,-6.5) circle [radius=0.08]; \draw  [fill=black](6.6,-7) circle [radius=0.08]; \draw  [fill=black](6.6,-6.5) circle [radius=0.08];\draw  [fill=black](6.6,-6.5) circle [radius=0.08]; \draw  [fill=red](7.4,-6.5) circle [radius=0.08]; \draw  [fill=black](7.4,-7.5) circle [radius=0.08]; \draw  [fill=blue](7.4,-7) circle [radius=0.08];

\node [above] at (1,-6){$v_{0}$};
\node [above] at (7,-6){$v_{n}$};

\end{tikzpicture}
\end{center}
\caption{Lobster}\label{lobster}
\end{figure}

\begin{exmp}
Let $T$ be a lobster tree. Then either $v_0$ or $v_n$, or both of them, only contain a single long leg, and the other major stems contain only midlegs and pendants. In this case,
$\beta(T)+\sum\limits_{i=1,m_{i}\geq 2}^{l}{m_{i}}-l\leq \beta(T^{3})\leq \beta(T)+\sum\limits_{i=1,m_{i}\geq 2}^{l}{m_{i}}-l +3.$ (see Figure \ref{lobster})
\end{exmp}

\begin{exmp}
Let $T$ be a spider tree. If it is a star, then $\beta{(T^{3})}=\beta{(T)}+1$, otherwise, we have $\beta(T)+\sum\limits_{i=1,m_{i}\geq 2}^{l}{m_{i}}-l\leq \beta(T^{3})\leq \beta(T)+\sum\limits_{i=1,m_{i}\geq 2}^{l}{m_{i}}-l+2$. (see Figure \ref{spider})
\end{exmp}
\begin{figure}
\begin{center}
\begin{tikzpicture}

\draw (-1,0)--(1,0);\draw (0,1)--(0,-1);

\draw  [fill=blue](0,0) circle [radius=0.07]; \draw  [fill=red](-1,0) circle [radius=0.07]; \draw  [fill=red](1,0) circle [radius=0.07]; \draw  [fill=black](0,1) circle [radius=0.07];\draw  [fill=red](0,-1) circle [radius=0.07];

\draw (2,0)--(4,0); \draw (3,1)--(3,-1);\draw (4,0)--(4,-1);

\draw [fill=red](2,0) circle [radius=0.07];\draw  [fill=red](3,-1) circle [radius=0.07];\draw [fill=red](3,1) circle [radius=0.07];\draw  [fill=blue](3,0) circle [radius=0.07];\draw  [fill=blue](4,0) circle [radius=0.07]; \draw  [fill=black](4,-1) circle [radius=0.07];

\draw (5,0)--(7,0);\draw (6,1)--(6,-1.5);\draw (7,0)--(7,-1.5);

\draw [fill=red](5,0) circle [radius=0.07];\draw  [fill=red](6,1) circle [radius=0.07];\draw [fill=black](6,-1.5) circle [radius=0.07];\draw [fill=black](6,-.75) circle [radius=0.07];\draw  [fill=black](6,0) circle [radius=0.07];\draw  [fill=blue](7,0) circle [radius=0.07]; \draw  [fill=blue](7,-.75) circle [radius=0.07];\draw [fill=red](7,-1.5) circle [radius=0.07];

\draw (8,0)--(10,0); \draw (9,1.5)--(9,-1.5);\draw (9,1.5)--(10,1.5);\draw (9,-1.5)--(8,-1.5); \draw (10,0)--(10,-1.5);\draw (8,0)--(8,1.5);

\draw [fill=black](8,0) circle [radius=0.07];\draw  [fill=black](8,0) circle [radius=0.07];\draw [fill=black](8,.75) circle [radius=0.07];\draw [fill=black](8,1.5) circle [radius=0.07];\draw [fill=black](9,0) circle [radius=0.07];\draw  [fill=blue](10,0) circle [radius=0.07];\draw  [fill=blue](10,-.75) circle [radius=0.07]; \draw  [fill=red](10,-1.5) circle [radius=0.07];\draw [fill=red](10,1.5) circle [radius=0.07]; \draw [fill=black](9,1.5) circle [radius=0.07]; \draw [fill=black](9,.75) circle [radius=0.07];\draw [fill=red](8,-1.5) circle [radius=0.07];\draw [fill=black](9,-1.5) circle [radius=0.07];\draw [fill=black](9,-.75) circle [radius=0.07];

\draw (11,0)--(15,0);\draw (12,1)--(12,-1.5);\draw (12,0)--(12.9,-1); \draw (12,0)--(12.9,1);\draw (12.9,1)--(13.9,1);

\draw [fill=red](11,0) circle [radius=0.07];\draw  [fill=red](12,1) circle [radius=0.07];\draw [fill=black](12,0) circle [radius=0.07];\draw [fill=blue](12,-.75) circle [radius=0.07];\draw [fill=red](12,-1.5) circle [radius=0.07];\draw  [fill=black](12.9,-1) circle [radius=0.07];\draw  [fill=black](12.45,-.5) circle [radius=0.07]; \draw  [fill=blue](13,0) circle [radius=0.07];\draw [fill=blue](14,0) circle [radius=0.07]; \draw [fill=red](15,0) circle [radius=0.07]; \draw [fill=black](12.45,.5) circle [radius=0.07];\draw [fill=black](12.9,1) circle [radius=0.07];\draw [fill=red](13.9,1) circle [radius=0.07];

\end{tikzpicture}
\end{center}
\caption{Spider}\label{spider}
\end{figure}

In a $d$-regular tree $T$, only pendants can be attached to every major stem. Let the length of a central path $P$ in $T$ be $2t$, where $t$ is the depth of $T$. Then the total number of pendants in $T$ is $d(d-1)^{t-1}$.

\begin{exmp}
Let $T$ be a $d$-regular tree $(d\geq 3)$ with depth $t$. If $t\leq 2$ then $\beta{(T)}\leq \beta{(T^{3})}\leq \beta{(T)}+d$ and for $t\geq 3$, $\beta{(T)}\leq \beta{(T^{3})}\leq \beta{(T)}+d(d-1)^{t-3}(d-2)$. (see Figure \ref{d-regular})
\end{exmp} 
  
\begin{figure}
\begin{center}
\begin{tikzpicture}


\draw (0,0)--(-2,-1);\draw (0,0)--(0,-1); \draw (0,0)--(2,-1);
\draw  [fill=black](0,0) circle [radius=0.08]; \draw  [fill=black](0,-1) circle [radius=0.08];\draw  [fill=black](2,-1) circle [radius=0.08]; \draw  [fill=black](-2,-1) circle [radius=0.08];

\draw (-2,-1)--(-2.5,-2);\draw (-2,-1)--(-1.5,-2);  \draw (0,-1)--(-.5,-2);\draw (0,-1)--(.5,-2);   \draw (2,-1)--(1.5,-2);\draw (2,-1)--(2.5,-2);
\draw  [fill=blue](-2.5,-2) circle [radius=0.08];\draw  [fill=black](-1.5,-2) circle [radius=0.08];\draw  [fill=blue](-.5,-2) circle [radius=0.08];\draw  [fill=black](.5,-2) circle [radius=0.08];\draw  [fill=blue](1.5,-2) circle [radius=0.08];\draw  [fill=black](2.5,-2) circle [radius=0.08];

\draw (-2.5,-2)--(-3,-3);\draw (-2.5,-2)--(-2.2,-3);     \draw (-1.5,-2)--(-1.8,-3);\draw (-1.5,-2)--(-1.2,-3);     \draw (-.5,-2)--(-.8,-3);\draw (-.5,-2)--(-.2,-3);      \draw (.5,-2)--(.2,-3);\draw (.5,-2)--(.8,-3);        \draw (1.5,-2)--(1.2,-3);\draw (1.5,-2)--(1.8,-3);      \draw (2.5,-2)--(2.2,-3);\draw (2.5,-2)--(2.8,-3);

\draw  [fill=red](-3,-3) circle [radius=0.08];\draw  [fill=black](-2.2,-3) circle [radius=0.08];\draw  [fill=red](-1.8,-3) circle [radius=0.08];\draw  [fill=black](-1.2,-3) circle [radius=0.08];\draw  [fill=red](-.8,-3) circle [radius=0.08];\draw  [fill=black](-.2,-3) circle [radius=0.08];\draw  [fill=red](.2,-3) circle [radius=0.08];\draw  [fill=black](.8,-3) circle [radius=0.08];\draw  [fill=red](1.2,-3) circle [radius=0.08];\draw  [fill=black](1.8,-3) circle [radius=0.08];\draw  [fill=red](2.2,-3) circle [radius=0.08];
\draw  [fill=black](2.8,-3) circle [radius=0.08];

\draw[<-] (3.5,0)--(3.5,-1.2);\draw[->] (3.5,-1.7)--(3.5,-3.1); \node at (3.5,-1.5){$t$-depth};
\draw (8,0)--(6,-1);\draw (8,0)--(8,-1); \draw (8,0)--(10,-1);
\draw  [fill=blue](8,0) circle [radius=0.08]; \draw  [fill=blue](6,-1) circle [radius=0.08];\draw  [fill=blue](8,-1) circle [radius=0.08]; \draw  [fill=black](10,-1) circle [radius=0.08];

\draw (6,-1)--(5.5,-2);\draw (6,-1)--(6.5,-2);  \draw (8,-1)--(7.5,-2);\draw (8,-1)--(8.5,-2);   \draw (10,-1)--(9.5,-2);\draw (10,-1)--(10.5,-2);
\draw  [fill=red](5.5,-2) circle [radius=0.08];\draw  [fill=black](6.5,-2) circle [radius=0.08];\draw  [fill=red](7.5,-2) circle [radius=0.08];\draw  [fill=black](8.5,-2) circle [radius=0.08];\draw  [fill=red](9.5,-2) circle [radius=0.08];\draw  [fill=black](10.5,-2) circle [radius=0.08];

\end{tikzpicture}
\end{center}
\caption{A $d$-regular tree}\label{d-regular}
\end{figure}
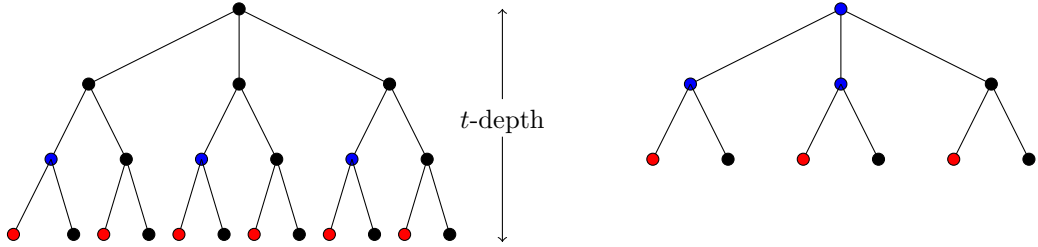

\section{Characterization of some restricted $T^{3}$ satisfying $\beta{(T^{3})}=\beta{(T)}$}

\begin{prop}\label{pair}
Let $T=(V, E)$ be a tree having at least two major stems. If $\beta{(T^{3})}=\beta{(T)}$, then there exists atleast one pair of major stems $v_{i},v_{j}$ satisfying $d_{T}(v_{i},v_{j})\equiv 1 \hspace{0.2em} \mbox{or} \hspace{0.2em} 2 \hspace{0.2em}( \mbox{mod} \hspace{0.2em} 3)$.
\end{prop}

\begin{proof}
On the contrary, let every pair of major stems have their distances as $0$ (mod $3$). Using Theorem \ref{0}, we observe that since $\beta{(T^{3})}=\beta{(T)}$, except for one, from all the legs of every major stem of $T$, we can pick at most one vertex for the metric basis of $T^{3}$.
Hence, from Theorem \ref{lowerbound} it can be easily verified that the number of midlegs attached to any major stem is at most one.

\vspace{0.3em}
\begin{claim}
\textit{To choose vertices for a metric basis $S$ of $T^{3}$, if we select a vertex from a long leg (or a midleg)  attached to any major stem $v$ (say), it is mandatory to choose the vertex which is at distance $0$ or $1$ (mod $3$) from $v$ on the same leg.} 
\end{claim}
\vspace{0.3em}

\begin{claimproof}
If we select a vertex (say $y$) in $S$ from a long leg/midleg $L$ attached to the major stem $v$ satisfying $d_{T}(v,y)\equiv 2$ (mod $3$), then the vertices $x$ and $x^{\prime}$ will possess the same code in $T^{3}$ measured from $y,z$ where $x,x^{\prime}$ are two neighbours of $v$ on the legs $L, L^{\prime}$ respectively, where $L^{\prime}$ is the leg that is left aside from vertex selection for $S$ and $z$ is a vertex from any branch of $v$ apart from $L$ and $L^{\prime}$. Hence $d_{T}(v,y)\equiv 0$ or $1$ (mod $3$).
\end{claimproof}

Using Claim $1$, we construct a vertex subset $S$ of $V$ by inserting a vertex from each leg (apart from one) of all the major stems that are at a distance of $0$ or $1$ (mod $3$) from the major stems.

\vspace{0.3em}
\begin{claim} \textit{There will always remain at least one pair of vertices in $T^{3}$ which can not be resolved by any vertex of $S$.} 
\end{claim}

\begin{claimproof}
Let $v_{1},v_{2}$ be two major stems satisfying $d_{T}(v_{1},v_{2})=3m$ for some integer $m$. Now consider the vertices $u_{0},v_{0}$ of an edge $e(=u_{0}v_{0})\in E$ on the intermediate path joining the vertices $v_{1},v_{2}$ in $T$ so that $d_{T}(v_{1},u_{0})\equiv 1$ (mod $3$) and $d_{T}(v_{2},v_{0})\equiv 1$ (mod $3$). Using the result of Claim $1$ one can verify that there is no vertex $x$ coming from the legs of $v_{1},v_{2}$, which can resolve $u,v$ as min $\{d_{T}(x,u_{0}), d_{T}(x,v_{0})\}\equiv 1$ or $2$ (mod $3$).
Similarly, it can be verified that $u_{0},v_{0}$ can not be resolved in $T^{3}$ by any $x$ coming from the legs of some other major stems that occur in the extended path of $v_{1}$ or $v_{2}$ as $d_{T}(v_{i},v_{j})\equiv 0$ (mod $3$) for all $v_{i}\neq v_{j}$. 

\vspace{0.3em}
\noindent Now we show that $u_{0},v_{0}$ can not be resolved by any vertex $x$ comes from the leg of a major stem $v_{3}$ that is connected with an intermediate vertex $s$ of the path joining $v_{1}, v_{2}$. For this, first, we consider the case when $s=u_{0}$ or $v_{0}$. Without loss of generality, if $v_{0}=s$ then $d_{T}(v_{3},s)\equiv 2$ (mod $3$) as $d_{T}(v_{3},v_{2})\equiv 0$ (mod $3$). Therefore, $d_{T}(v_{1},v_{3})=d_{T}(v_{1},u_{0})+d_{T}(u_{0},v_{0})+d_{T}(v_{0},v_{3})\equiv 1+1+2$ (mod $3$) $\equiv 1$ (mod $3$). This introduces a contradiction. 
Next, we consider the case when $s\neq u,v$. Without loss of generality, we assume min $\{d_{T}(v_{3},v_{0}),d_{T}(v_{3},u_{0})\}=d_{T}(v_{3},v_{0})$. Therefore, $s$ must lie within the intermediate path of $v-v_{2}$. 

\vspace{0.4em}
\noindent If $d_{T}(v_{3},s)\equiv 1$ (mod $3$) then $d_{T}(s,v_{2})\equiv 2$ (mod $3$) as $d_{T}(v_{2},v_{3})\equiv 0$ (mod $3$). Hence, $d_{T}(v_{0},s)=d_{T}(v_{0},v_{2})-d_{T}(s,v_{2})\equiv 2$ (mod $3$). Therefore, $d_{T}(v_{1},v_{3})=d_{T}(v_{1},u_{0})+d_{T}(u_{0},v_{0})+d_{T}(v_{0},s)+d_{T}(s,v_{3})\equiv 1+1+2+1$ 
(mod $3$) $\equiv 2$ (mod $3$), which is not true as per our assumption. 

\vspace{0.3em}
\noindent  If $d_{T}(v_{3},s)\equiv 2$ (mod $3$), then we get $d_{T}(v_{1},v_{3})\equiv 1$ (mod $3$), therefore a similar contradiction arises.

\vspace{0.3em}
\noindent If $d_{T}(v_{3},s)\equiv 0$ (mod $3$), then $d_{T}(v_{0},s)=d_{T}(v_{0},v_{2})-d_{T}(s,v_{2})=1-0$ (mod $3$) $\equiv 1$ (mod $3$). Hence, $d_{T}(v_{0},v_{3})=d_{T}(v_{0},s)+d_{T}(s,v_{3})=1+0$ (mod $3$)$\equiv 1$ (mod $3$).

\vspace{0.3em}
Therefore, following Lemma \ref{3} one can verify that $u_{0},v_{0}$ can not be resolved by any vertex $x\in S$ coming from the legs of $v_{3}$ as $d_{T}(x,v_{3})\equiv 0$ or $1$ (mod $3$) from Claim $1$ implies $d_{T}(x,v_{0})\equiv 1$ or $2$ (mod $3$).
\end{claimproof}

\noindent Since $\beta{(T^{3})}=\beta{(T)}$, any metric basis of $T^{3}$ can only be constructed in the above way, as we did for $S$. But from Claim $2$ we will always get a pair of vertices in $T^{3}$ that can not be resolved by any vertex of $S$. Hence, we get a contradiction. Therefore, we will always get a pair of major stems (say $\{v_{i},v_{j}\}$) satisfying $d_{T}(v_{i},v_{j})\equiv 1$ or $2$ (mod $3$).
\end{proof}





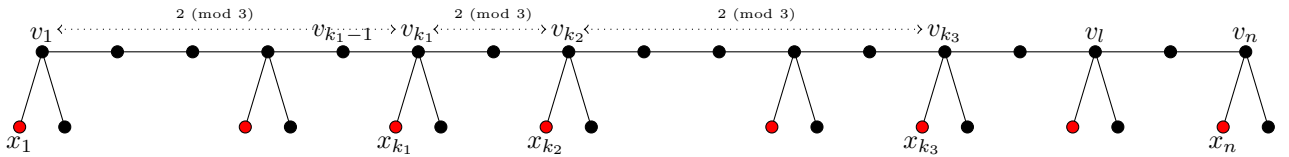
\begin{figure}[ht]
\begin{center}
\begin{tikzpicture}


\draw[<-] [dotted] (0.2,0.3)--(4,0.3);
\draw[->] [dotted] (4,0.3)--(4.7,0.3);
\node at  (2.3,0.5) {\tiny{2 (mod $3$)}};

\draw[<-] [dotted] (5.2,0.3)--(6,0.3);
\draw[->] [dotted] (6,0.3)--(6.7,0.3);

\node at (6,0.5){\tiny{2 (mod $3$)}};

\draw[<-] [dotted] (7.2,0.3)--(8.3,0.3);
\draw[->] [dotted] (8.3,0.3)--(11.7,0.3);

\node at (9.5,0.5){\tiny{2 (mod $3$)}};

\draw (0,0)--(16,0);
\draw (0,0)--(-0.3,-1); \draw (0,0)--(0.3,-1);
\draw (3,0)--(2.7,-1); \draw (3,0)--(3.3,-1);
\draw (5,0)--(4.7,-1); \draw (5,0)--(5.3,-1);
\draw (7,0)--(6.7,-1); \draw (7,0)--(7.3,-1);
\draw (10,0)--(9.7,-1); \draw (10,0)--(10.3,-1);
\draw (12,0)--(11.7,-1); \draw (12,0)--(12.3,-1);
\draw (14,0)--(13.7,-1); \draw (14,0)--(14.3,-1);
\draw (16,0)--(15.7,-1); \draw (16,0)--(16.3,-1);

\node [above] at (0,0) {$v_1$}; \node [above] at (5,0) {$v_{k_1}$};

\node[above] at (4,0) {$v_{k_{1}-1}$};
\node [above] at (7,0) {$v_{k_2}$};
\node [above] at (12,0) {$v_{k_3}$};
\node [above] at (14,0) {$v_l$};
\node [above] at (16,0) {$v_n$};
\node [below] at (-0.3,-1) {$x_1$};
\node [below] at (4.7,-1) {$x_{k_1}$};
\node [below] at (6.7,-1) {$x_{k_2}$};
\node [below] at (11.7,-1) {$x_{k_3}$};
\node [below] at (15.7,-1) {$x_n$};

\draw  [fill=black](0,0) circle [radius=0.08]; \draw  [fill=black](1,0) circle [radius=0.08];\draw  [fill=black](2,0) circle [radius=0.08]; \draw  [fill=black](3,0) circle [radius=0.08]; \draw  [fill=black](4,0) circle [radius=0.08];\draw  [fill=black](5,0) circle [radius=0.08];\draw  [fill=black](6,0) circle [radius=0.08];\draw  [fill=black](7,0) circle [radius=0.08];\draw  [fill=black](8,0) circle [radius=0.08];\draw  [fill=black](9,0) circle [radius=0.08];\draw  [fill=black](10,0) circle [radius=0.08];\draw  [fill=black](11,0) circle [radius=0.08];\draw  [fill=black](12,0) circle [radius=0.08];\draw  [fill=black](13,0) circle [radius=0.08];\draw  [fill=black](14,0) circle [radius=0.08];\draw  [fill=black](15,0) circle [radius=0.08];\draw  [fill=black](16,0) circle [radius=0.08];

\draw  [fill=red](-0.3,-1) circle [radius=0.08]; \draw  [fill=black](0.3,-1) circle [radius=0.08];\draw  [fill=red](2.7,-1) circle [radius=0.08]; \draw  [fill=black](3.3,-1) circle [radius=0.08]; \draw  [fill=red](4.7,-1) circle [radius=0.08];\draw  [fill=black](5.3,-1) circle [radius=0.08];\draw  [fill=red](6.7,-1) circle [radius=0.08];\draw  [fill=black](7.3,-1) circle [radius=0.08];\draw  [fill=red](9.7,-1) circle [radius=0.08];\draw  [fill=black](10.3,-1) circle [radius=0.08];\draw  [fill=red](11.7,-1) circle [radius=0.08];\draw  [fill=black](12.3,-1) circle [radius=0.08];\draw  [fill=red](13.7,-1) circle [radius=0.08];\draw  [fill=black](14.3,-1) circle [radius=0.08];\draw  [fill=red](15.7,-1) circle [radius=0.08];\draw  [fill=black](16.3,-1) circle [radius=0.08];
\end{tikzpicture}
\end{center}
\caption{Trees satisfying $\beta{(T^{3})}=\beta{(T)}$}\label{specialtree}
\end{figure}

Below, we characterize those cube of trees that possess all their stems on one of their central paths, \footnote{also, known as diametral paths}, stems contain pendants only as their legs and have their metric dimension similar to the metric dimension of their associated trees.

\begin{thm}
Let $T=(V, E)$ be a tree where every stem contains pendants only as their legs in $T$.  If all the stems lie on a central path, $P=(x_{1}, v_{1},v_{2},\hdots,v_{n},x_{n})$
\footnote{$v_{1},\hdots v_{n}$ are path vertices and $x_{1},x_{n}$ be pendants attached to $v_{1},v_{n}$ respectively} of $T$ (see Figure \ref{specialtree}), then $\beta{(T^{3})}=\beta{(T)}$ if and only if the following conditions are satisfied:.

\begin{enumerate}
\item There are atleast three major stems $v_{k_{1}}, v_{k_{2}},v_{k_{3}}$ between $v_{1},v_{n}$ along $P$ such that $d_{T}(v_{1},v_{k_{1}})\equiv 2$ $(\text{mod} \hspace{0.3em} 3)$, $d_{T}(v_{1},v_{k_{2}})\equiv 1$ $(\text{mod} \hspace{0.3em} 3)$, $d_{T}(v_{1},v_{k_{3}})\equiv 0$ $(\text{mod} \hspace{0.3em} 3)$.
Distances between 
any of the above pairs considered to be minimum satisfying the above criteria.

\item There does not exist any stem between $v_{1},v_{k_{1}} (v_{k_{3}}, v_{n})$, which is at $1$ $(\text{mod} \hspace{0.3em} 3)$ distance from $v_{1} (v_{n}$) along $P$. 


\item There can not exist any pair of stems $\{v_{m},v_{k}\}$ such that $d_{T}(v_{m},v_{k})\equiv 1 $ $(\text{mod} \hspace{0.3em} 3)$ and $d_{T}(v_{k},v_{r})\equiv 0$ $(\text{mod} \hspace{0.3em} 3)$ for all major stems $v_{r}$ satisfying $v_{k_{2}}<v_{m}<v_{k}\leq v_{r}\leq v_{n}$  on $P$ where $v_{k}$ is the minimum distance major stem from $v_{m}$ along $P$. 
\end{enumerate}
\end{thm}

\begin{proof}
Let $\beta{(T^{3})}=\beta{(T)}$. Then any metric basis $S$ of $T^{3}$ must contain one except all the pendants attached to every major stem of the tree $T$ by Corollary \ref{30}. Since $T$ contains legs as only pendants, $v_{1},v_{n}$ are the first and last major stems on the central path $P$. 


\vspace{0.4em}
\noindent \textit{proof of condition $1$ and $2$.}
To resolve $v_{1}, x_{1}$ there must exist a pendant $x_{k_{1}}\in S$ such that $d_{T}(v_{1},x_{k_{1}})\equiv0$ (mod $3$) by condition $1$ of Theorem \ref{neccsuff}. Hence, without loss of generality, we choose $v_{k_{1}}$ to be the minimum distance major stem from $v_{1}$ satisfying $d_{T}(v_{1},v_{k_{1}})\equiv 2$ (mod $3$). Let $v_{k_{1}-1}$ be the neighbour of the major stem $v_{k_{1}}$ satisfying $v_{1}<v_{k_{1}-1}<v_{k_{1}}$ along $P$. Now to resolve $v_{k_{1}-1}, v_{k_{1}}$, we need a major stem $v_{k_{2}}$ and its pendant $x_{k_{2}}\in S$ such that min $\{d_{T}(v_{k_{1}},x_{k_{2}}),d_{T}(v_{k_{1}-1},x_{k_{2}})\equiv 0$ (mod $3$). If $v_{1}<v_{k_{2}}<v_{k_{1}}$ on $P$, then $d_{T}(v_{1},v_{k_{2}})=d_{T}(v_{1},v_{k_{1}})-(d_{T}(v_{k_{1}},v_{k_{1}-1})+d_{T}(v_{k_{1}-1},v_{k_{2}})\equiv 2$ (mod $3$), which is not possible by the choice of $v_{k_{1}}$. Therefore, $v_{k_{2}}>v_{k_{1}}$ on $P$. We consider $v_{k_{2}}$ to be the minimum distance major stem from $v_{k_{1}}$ satisfying $d_{T}(v_{k_{1}},v_{k_{2}})\equiv d_{T}(v_{1},v_{k_{1}})+d_{T}(v_{k_{1}},v_{k_{2}})\equiv 2$ (mod $3$).

\vspace{0.7em}
Let there be a stem $v_{m}$ between $v_{1}, v_{k_{1}}$ satisfying $d_{T}(v_{1},v_{m})\equiv 1$ (mod $3$) and let $x_{m}$ be the pendant of $v_{m}$ which is not in $S$. Then, to resolve $x_{m},v_{m-1}$, we need a pendant $x_{p}\in S$ attached to some major stem $v_{p}$ within $v_{1},v_{m}$ along $P$ such that $d_{T}(x_{p},v_{m-1})\equiv 0 $ or $2$ (mod $3$) by condition \ref{two distance} of Corollary \ref{distance}.
Therefore, $d_{T}(v_{p},v_{m-1})\equiv 1 \hspace{0.3em} \text{or} \hspace{0.3em} 2$ (mod $3$). If $d_{T}(v_{p},v_{m-1})\equiv 1 $ (mod $3$), we get $d_{T}(v_{1},v_{p})=d_{T}(v_{1},v_{m})-(d_{T}(v_{p},v_{m-1})+d_{T}(v_{m-1},v_{m}))\equiv 2$ (mod $3$), which contradicts the choice of $v_{k_{1}}$. 
Again, if $d_{T}(v_{p},v_{m-1})\equiv 2$ (mod $3$) then we get $d_{T}(v_{1},v_{p})\equiv 1$ (mod $3$).
Since $v_{1}<v_{p}<v_{k_{1}}$ and $d_{T}(v_{1},v_{p})\equiv 1$ (mod $3$), proceeding similarly as above, contradiction arises after finite steps when we get the minimum distance stem at $1$ (mod $3$) distance from $v_{1}$.
 
\vspace{0.4em}
Further, to resolve $v_{k_{2}}, v_{k_{2}-1}$, there must exist some pendant $x_{k_{3}}\in S$ attached to some major stem $v_{k_{3}}$ such that min $\{d_{T}(x_{k_{3}},v_{k_{2}}), d_{T}(x_{k_{3}}, v_{k_{2}-1})\}\equiv 0$ (mod $3$). If $v_{k_{1}}<v_{k_{3}}<v_{k_{2}}$ along $P$, then $d_{T}(v_{k_{1}},v_{k_{3}})=d_{T}(v_{k_{1}},v_{k_{2}})-d_{T}(v_{k_{2}},v_{k_{3}})\equiv 2$ (mod $3$), which contradicts the choice of $v_{k_{2}}$. 
Again, if $v_{1}<v_{k_{3}}<v_{k_{1}}$, then $d_{T}(v_{1},v_{k_{3}})=d_{T}(v_{1},v_{k_{1}})-d_{T}(v_{k_{1}},v_{k_{3}})\equiv 1$ (mod $3$), which is not possible from the above paragraph. Therefore, $v_{k_{2}}<v_{k_{3}}<v_{n}$, which imply $d_{T}(v_{1},v_{k_{3}})=d_{T}(v_{1},v_{k_{2}})+d_{T}(v_{k_{2}},v_{k_{3}})\equiv 0$ (mod $3$).  We consider $v_{k_{3}}$ to be the minimum distance major stem from $v_{k_{2}}$ satisfying $d_{T}(v_{k_{2}},v_{k_{3}})\equiv 2$ (mod $3$).

\vspace{0.4em}
One can also verify with a similar approach as we did earlier and find that if $v_{l}$ is the minimum distance stem from $v_{n}$ satisfying $d_{T}(v_{l},v_{n})\equiv 1$ (mod $3$) where $v_{k_{3}}\leq v_{l}<v_{n}$. Then $d_{T}(v_{l},v_{n})$ must be equivalent to $0$ or $2$ (mod $3$)
   
\vspace{0.4em}
\noindent \textit{proof of condition $3$.} Let there exist a pair of stems $\{v_{m},v_{k}\}$ satisfying $d_{T}(v_{m},v_{k})\equiv 1$ (mod $3$) and $d_{T}(v_{k},v_{r})\equiv 0$ (mod $3$) for all major stems $v_{r}$ such that $v_{k_{2}}<v_{m}<v_{k}\leq v_{r}\leq v_{n}$ and $v_{k}$ is the minimum distance 
major stem from $v_{m}$ along $P$. Let $v_{m+1}$ be the neighbour of $v_{m}$ satisfying $v_{m}<v_{m+1}\leq v_{k}$ along $P$. Then $d_{T}(v_{m+1},v_{k})\equiv 0$ (mod $3$), which imply $d_{T}(v_{m+1},v_{r})\equiv 0$ (mod $3$) and hence $d_{T}(v_{m+1},x_{r})\equiv 1$ (mod $3$), where $x_{r}$ be any pendant attached to $v_{r}$ that is in $S$. Let $x_{m}$ be a pendant of $v_{m}$ that is not in $S$, then $x_{m},v_{m+1}$ can not be resolved by any pendant of $S$ since $d_{T}(v_{k},v_{r})\equiv 0$ (mod $3$) for all $v_{r}$ satisfying $v_{k}\leq v_{r}\leq v_{n}$. Hence, following Lemma \ref{3} we get a contradiction as $\beta{(T^{3})}=\beta{(T)}$. Therefore, the result follows.

\vspace{0.4em}
\noindent \textit{Sufficient:}  We consider $T$ to be a tree that satisfies all the given conditions. From Corollary \ref{1.5} it is already known that $\beta{(T^{3})}\geq \beta{(T)}$. Hence, to prove $\beta{(T^{3})}=\beta{(T)}$, it is sufficient to show that $\beta{(T^{3})}\leq \beta{(T)}$. Let $S$ be any metric basis of $T$, then all except one pendant from every major stem of $T$ are the only members of $S$ by Theorem \ref{0}. Now we show that $S$ is a resolving set of $T^{3}$ also.
For this, it is sufficient to prove that any two vertices $u,v\in V\setminus S$ can be resolved by at least one vertex of $S$. Let $x_{1},x_{k_{1}},x_{k_{2}},x_{k_{3}},x_{n}$ be the pendants attached to the major stems $v_{1},v_{k_{1}},v_{k_{2}},v_{k_{3}}, v_{n}$ respectively, which are included in $S$. (see Figure \ref{specialtree})
\vspace{0.2em}

i) If both $u$ and $v$ appear on the central path $P$ or one among them attached to a stem on $P$, then at least one among $x_{1},x_{k_{1}}, x_{k_{2}}, x_{k_{3}},x_{n}$ resolves $u,v$ in $T^{3}$ using Theorem \ref{neccsuff}. 

\vspace{0.4em}

When $d_{T}(u,v)\neq 2$, then the existence of the above pendants is ensured by condition $1$. Next, we consider the situation when $d_{T}(u,v)=2$ and $u$ is a pendant attached to some stem $v_{p}$ and $v$ is on the central path $P$.
Let $v=v_{p+1}$ and $v_{k}$ be the minimum distance major stem from $v_{n}$. Then $v_{p}\leq v_{k}$ on $P$. From condition $2$, it is known that $d_{T}(v_{k},v_{n})\equiv 0 \hspace{0.3em} \mbox{or} \hspace{0.3em} 2$ (mod $3$). When $v_{k_{2}}<v_{p}\leq v_{n}$, then if $d_{T}(v_{p},v_{k})\equiv 1$ (mod $3$), then by using condition $3$, we get $d_{T}(v_{k},v_{n})\equiv 2$ (mod $3$). Therefore, $d_{T}(v_{p+1},x_{n})\equiv 2$ (mod $3$). Hence, $x_{n}$ resolves $u,v$ in $T^{3}$ by Lemma \ref{3}. In other situations, i.e., when $d_{T}(v_{p},v_{k})\equiv 0$ or $2$, then $x_{k}$ or $x_{n}$ resolves $u,v$.
Again, if $v_{1}\leq v_{p}\leq v_{k_{1}}$ ($v_{k_{1}}<v_{p}\leq v_{k_{2}}$), then either $x_{k_{1}}$ or $x_{k_{2}}$ ($x_{k_{2}}$ or $x_{k_{3}}$) resolves $u,v$ in $T^{3}$ by Lemma \ref{3}. 

\vspace{0.4em}
Let $v=v_{p-1}$. Then, by Lemma \ref{3} at least one among $x_{k_{1}},x_{1}$ resolves $u,v$ when $v_{k_{1}}\leq v_{p}\leq v_{n}$ along $P$. Again, if $v_{1}<v_{p}< v_{k_{1}}$, then $d_{T}(v_{1},v_{p})\not\equiv 2$ (mod $3$) from the definition of $v_{k_{1}}$. Further, $d_{T}(v_{1},v_{p})\not\equiv 1$ (mod $3$) by condition $2$.  Hence $d_{T}(v_{1},v_{p})\equiv 0$ (mod $3$) and therefore $d_{T}(x_{1},v_{p-1})=d_{T}(x_{1},v_{1})+d_{T}(v_{1},v_{p-1})\equiv 0$ (mod $3$). Hence, $x_{1}$ resolves $u,v$ in $T^{3}$ by Lemma \ref{3}. 
\vspace{0.4em}

ii) If both $u,v$ are pendants attached to two different stems, $v_{m_{1}}, v_{m_{2}}$ respectively. Let $v_{m_{1}}<v_{m_{2}}$ on $P$. 

\vspace{0.2em}
First, we consider the situation when $d_{T}(u,v)=3$. If $d_{T}(v_{1},v_{m_{1}})\equiv 1$ (mod $3$), then $x_{1}$ resolves $u,v$ by condition $3$ of Theorem \ref{neccsuff} since min$\{d_{T}(x_{1},u), d_{T}(x_{1},v)\}\equiv 0$ (mod $3$).

Next, if $d_{T}(v_{1},v_{m_{1}})\equiv 2$ (mod $3$), then either $v_{k_{1}}>v_{m_{2}}$ or $v_{k_{1}}< v_{m_{2}}$ on $P$. If $v_{k_{1}}>v_{m_{2}}$, then
$d_{T}(v_{k_{1}},v_{m_{2}})=d_{T}(v_{1},v_{k_{1}})-d_{T}(v_{1},v_{m_{1}})-d_{T}(v_{m_{1}},v_{m_{2}})\equiv 2$ (mod $3$). Therefore, 
min $\{d_{T}(x_{k_{2}},v),d_{T}(x_{k_{2}},u)\}=d_{T}(x_{k_{2}},v)=d_{T}(x_{k_{2}},v_{k_{2}})+d_{T}(v_{k_{2}},v_{k_{1}})+d_{T}(v_{k_{1}},v_{m_{2}})+d_{T}(v_{m_{2}},v)\equiv 0$ (mod $3$). Hence, $x_{k_{2}}$ resolves $u,v$ in $T^{3}$. Again, if $v_{k_{1}}<v_{m_{2}}$, $d_{T}(v_{k_{1}},v_{m_{1}})=d_{T}(v_{1},v_{m_{1}})-d_{T}(v_{1},v_{k_{1}})\equiv 0$ (mod $3$) and hence it is easy to verify that $x_{k_{2}}$ resolves $u,v$ by condition $3$ of Theorem \ref{neccsuff}. 

Again, if $d_{T}(v_{1},v_{m_{1}})\equiv 0$ (mod $3$), then either $v_{k_{3}}\leq v_{m_{1}}$ or $v_{m_{1}}<v_{k_{3}}$. 
If $v_{k_{3}}\leq v_{m_{1}}$, then $x_{k_{1}},x_{1},x_{k_{2}}$ resolves $u,v$ depending on the situations $d_{T}(v_{k_{3}},m_{1})\equiv 0, 1,2$ (mod $3$) respectively. If $v_{m_{1}}<v_{k_{3}}$, then if   $v_{k_{3}}>v_{m_{2}}$, i.e., $v_{1}\leq v_{m_{1}}<v_{m_{2}}<v_{k_{3}}$ since $d_{T}(v_{1},v_{k_{3}})\equiv 0$ (mod $3$). It is easy to note that $v_{k_{1}}\neq v_{m_{1}}$ and min $\{d_{T}(x_{k_{1}},u),d_{T}(x_{k_{1}},v)\}\equiv 0$ (mod $3$). Hence, $x_{k_{1}}$ resolves $u,v$ in $T^{3}$ by condition $3$ of Theorem \ref{neccsuff}.

\vspace{0.2em}
If $d_{T}(u,v)=4$, then $x_{1}$ resolves $u,v$ when $d_{T}(v_{1},v_{m_{1}})\equiv 0$ or $1$ (mod $3$), otherwise, $x_{k_{1}}$ resolves $u,v$ in $T^{3}$ by condition $4$ of Theorem \ref{neccsuff}. Again, if $d_{T}(u,v)=5$, then $|d_{T}(x_{1},u)-d_{T}(x_{1},v)|=5$ or $|d_{T}(x_{1},u)-d_{T}(x_{1},v)|=3$ as per the situation $v_{m_{1}}\neq v_{1}$ or $v_{m_{1}}=v_{1}$.
Hence $x_{1}$ resolves $u,v$ in both circumstances by condition $5$ of Theorem \ref{neccsuff}. 
\end{proof}

 \section{Conclusion}
In this article, we have determined the necessary and sufficient conditions for a resolving set to be a metric basis for the cube of trees. Also, we developed the upper and lower bounds of the metric dimension of the same graph class. Further, we discuss the characterization of some restricted class of cube of trees satisfying $\beta{(T^{3})}=\beta{(T)}$. The following open problems are immediate from our study:

\vspace{0.7em}
\textbf{Problem $1.$} \textit{Find the bounds of the metric dimension of $T^{r}$ for any positive integer $r\geq 4$}. 

\vspace{0.4em}
\textbf{Problem $2$.}\textit{ Characterize the class of cube of trees that satisfy $\beta{(T^{r})}=\beta{(T)}$} for any positive integer $r$.

\end{document}